\documentclass[11pt]{article}
\usepackage[margin=1in]{geometry}

\usepackage{setspace}
\usepackage{import}

\usepackage{amsmath, amssymb, bbold, mathtools}
\usepackage[ruled, vlined]{algorithm2e}
\usepackage{multirow}
\usepackage{graphicx}
\usepackage{tcolorbox}
\usepackage{bbding}
\usepackage[export]{adjustbox}
\usepackage{subcaption}
\usepackage{xparse}

\usepackage{wasysym}
\usepackage{lmodern} 

\usepackage{natbib}
 \bibpunct[, ]{(}{)}{,}{a}{}{,}

\usepackage{tikz}
\usetikzlibrary{calc}
\usetikzlibrary{shapes}

\usepackage{amsthm}
\theoremstyle{plain}

\makeatletter

\usepackage{hyperref}
\usepackage[capitalise]{cleveref}

\makeatother
\newcounter{parentnumber}

\newtheorem{theorem}{Theorem}[section]
\newtheorem{lemma}[theorem]{Lemma}

\newtheorem{corollary}[theorem]{Corollary}
\newtheorem{remark}[theorem]{Remark}
\newtheorem{example}[theorem]{Example}
\newtheorem{proposition}[theorem]{Proposition}

\newtheorem{definition}[theorem]{Definition}

\crefname{theorem}{Theorem}{Theorems}
\crefname{lemma}{Lemma}{Lemmas}
\crefname{proposition}{Proposition}{Propositions}
\crefname{corollary}{Corollary}{Corollaries}
\crefname{definition}{Definition}{Definitions}

\usepackage{acronym}
\acrodef{wlog}[WLOG]{without loss of generality}
\acrodef{lsc}[lsc]{lower semi-continuous}

\definecolor{red}{RGB}{163, 31, 52}
\definecolor{gray}{RGB}{194, 192, 191}
\definecolor{blue}{RGB}{59, 89, 152}
\definecolor{green}{RGB}{0, 179, 0}

\newcommand{\norm}[1]{\left\|#1\right\|}
\newcommand{\tnorm}[1]{\|#1\|}

\newcommand{\abs}[1]{\left|#1\right|}

\newcommand{\E}[2]{{\mathbb E}_{#1} \left[ #2 \right]}

\newcommand{\one}[1]{\mathbb{1}\left\{#1\right\}}
\renewcommand{\tfrac}[2]{{#1}/{#2}}

\newcommand{\set}[2]{\left\{ #1\ : \ #2 \right\}}
\newcommand{\tset}[2]{\{ #1\ : \ #2 \}}
\newcommand{\KL}{{\rm{KL}}}
\newcommand{\tpose}{^\top}
\newcommand{\defn}[0]{:=}

\newcommand{\mean}{\texttt{mean}}
\newcommand{\median}{\texttt{median}}
\newcommand{\huber}{\texttt{huber}}
\newcommand{\phihuber}{\varphi^\huber}

\newcommand{\mc}{\mathcal}
\newcommand{\mb}{\mathbb}

\newcommand{\mr}{\mathrm}

\renewcommand{\d}{{\mathrm{d}}}

\renewcommand{\Re}{\mathrm{R}}

\newcommand{\st}{\mr{s.t.}}

\renewcommand{\emph}{\textbf}

\DeclareMathOperator{\cl}{cl}

\DeclareMathOperator{\TV}{TV}
\DeclareMathOperator{\D}{TV}
\DeclareMathOperator{\LP}{LP}

\DeclareMathOperator{\radius}{radius}
\DeclareMathOperator{\cheb}{center}

\DeclareMathOperator{\interior}{int}

\DeclareMathOperator{\Prob}{Prob}

\DeclareMathOperator{\sign}{sign}

\newcommand{\cX}{\mathcal{X}}

\newcommand{\cP}{\mathcal{P}}

\newcommand{\cD}{\mathcal{D}}

\newcommand{\Eb}{\mathbb{E}}
\newcommand{\Pb}{\mathbb{P}}
\newcommand{\Qb}{\mathbb{Q}}

\newcommand{\SpKL}[2]{S_{#1,#2}}

\newcommand{\Ptrue}{\mb P ^\star}
\newcommand{\thetatrue}{{\theta^\star}}
\newcommand{\Pcor}{\mb P ^{\mathrm{c}}}
\newcommand{\hPb}{\hat{\mb P}}

\newcommand{\resolution}{statistical resolution }
\newcommand{\rad}[3]{\underline{\Delta}(#1;#2,#3)}
\newcommand{\KLrad}[2]{\Delta_{\mathrm{wc}}(#1,#2)}
\newcommand{\KLradp}[2]{\Delta_{\mathrm{as}}(#1,#2)}
\renewcommand{\r}[0]{r_{\mathrm{wc}}}
\newcommand{\rp}[0]{r_{\mathrm{as}}}

\usepackage{rotating}
\usepackage{authblk}

\usepackage[appendix=append]{apxproof}

\allowdisplaybreaks

\title{From Distributional Robustness to Robust Statistics: A Confidence Sets Perspective}
\author[1]{Gabriel Chan\thanks{gabriel.chan@polytechnique.edu}}
\author[2]{Bart P.G.\ Van Parys\thanks{bart.vanparys@cwi.nl}}
\author[3]{Amine Bennouna\thanks{amineben@mit.edu}}
\affil[1]{\'Ecole Polytechnique}
\affil[2]{CWI Amsterdam}
\affil[3]{Massachusetts Institute of Technology}

\date{}

\begin{document}

\maketitle

\begin{abstract}
  We establish a connection between distributionally robust optimization (DRO) and classical robust statistics. We demonstrate that this connection arises naturally in the context of estimation under data corruption, where the goal is to construct ``minimal'' confidence sets for the unknown data-generating distribution. Specifically, we show that a DRO ambiguity set, based on the Kullback-Leibler divergence and total variation distance, is uniformly minimal, meaning it represents the smallest confidence set that contains the unknown distribution with at a given confidence power. Moreover, we prove that when parametric assumptions are imposed on the unknown distribution, the ambiguity set is never larger than a confidence set based on the optimal estimator proposed by \cite{huber1964robust}. This insight reveals that the commonly observed conservatism of DRO formulations is not intrinsic to these formulations themselves but rather stems from the non-parametric framework in which these formulations are employed.
\end{abstract}

\section{Introduction}
\label{sec:introduction}

Consider the stochastic optimization problem
\begin{equation}
  \label{eq:stochastic-optimization}
  \min_{x\in \mc X} \E{\Ptrue}{\ell(x, \xi)}
\end{equation}
where $\ell: \mc X\times\Xi\to\Re$ is a given loss function and $\mb P^\star\in \mc P$ denotes the distribution of the uncertainty $\xi\in\Xi$. Stochastic optimization is a well established framework for decision-making under uncertainty \citep{kleywegt2001stochastic}. However, in practice the distribution $\Ptrue$ is typically unobserved and hence needs to be estimated from historical data points $\{\xi_1, \dots, \xi_n\}$ with empirical distribution $\mb P_n$.

A natural data-driven counterpart to problem \eqref{eq:stochastic-optimization} is the so-called sample average approximation \citep{kleywegtSampleAverageApproximation2002} otherwise known as empirical risk minimization, where the unknown distribution $\Ptrue$ is replaced by the empirical distribution $\mb P_n$, and the data-driven formulation $\min_{x\in \mc X} \E{\mb P_n}{\ell(x, \xi)}$ is solved instead. However, blindly optimizing the empirical risk based on the empirical distribution $\mb P_n$ might (and often does) lead to decisions with poor performance under the actual but unknown distribution $\Ptrue$ \citep{smith2006optimizer}. In fact, several factors can lead to a substantial difference between the out-of-sample distribution and the distribution of the observed data samples. First, the distribution $\mb P_n$ is merely an empirical estimate subject to the randomness of the draw.
We can think of the difference attributed to this effect as statistical error which fades asymptotically and is expected to be of order $\mc O(1/\sqrt{n})$ based on the central limit theorem \citep{blanchetStatisticalLimitTheorems2023a}. A second factor is that the distribution generating the samples is not necessarily the same as the actual distribution $\Ptrue$ out-of-sample due for instance to data corruption, distribution shift or adversarial contamination. Unlike statistical error the effect of adversarial corruption typically does not disappear simply by collecting more data.

Distributionally robust optimization (DRO) offers a disciplined approach when making decisions based on historical data and comes with strong statistical guarantees. In a classical robust optimization approach \citep{bertsimas2018data, delage2010distributionally}, one typically proceeds in two steps. First, one tries to construct a confidence set $\mc S(\mb P_n)\subseteq \mc P$ of distributions which despite the finite number of collected data points tries to fence in the unknown distribution $\mb P^\star$. Typically, one seeks covering guarantees of the form $\mb P^\star \in \mc S(\mb P_n)$ with high probability over the randomness of the draw. Subsequently, a robust formulation
\begin{equation}
  \label{eq:dro}
  \min_{x\in \cal X} \max_{\mb Q\in \mc S(\mb P_n)}\E{\mb Q}{\ell(x, \xi)}
\end{equation}
guarantees a minimal worst-case cost in face of the remaining uncertainty as represented by the considered confidence set. From this perspective, it is clear that smaller confidence sets, provided they enjoy with the same covering guarantee, are preferable as they allow for less conservative downstream decisions in the robust formulation (\ref{eq:dro}).
\cite{lamRecoveringBestStatistical2019} and \cite{duchi2021statistics} point out that covering guarantees are not necessary (albeit sufficient) for the robust decision in formulation (\ref{eq:dro}) to have desirable statistical guarantees in particular when working with confidence sets characterized as $f$-divergence balls \citep{pardo2018statistical} around the empirical distribution. A similar observations has recently been made in the context of Wasserstein based confidence sets as well by \citet{shafieezadeh2019regularization, blanchet2019robust, gao2024wasserstein, blanchet2022confidence}.
However, empirical formulations where the confidence set $\mc S(\mb P_n)$ does not contain $\mb P^\star$ with high probability typically require additional complexity conditions on the loss function class $\set{\xi \mapsto\ell(x, \xi)}{x\in \cal X}$.
As the statistical properties of the robust formulation (\ref{eq:dro}) with a covering guarantee are entirely attributed to what confidence set estimator $\mc S$ is considered, we focus here on the covering guarantees, which allows us to make abstraction of the underlying loss function entirely, as also pointed out in recent work by \cite{liu2023smoothed}.

A plethora of different confidence sets with desirable properties have been considered in prior work. In early work by \cite{delage2010distributionally, bertsimas2005optimal, wiesemann2014distributionally} the confidence sets were moment based and selected because of their computational tractability.
Later on confidences sets based on a statistical distance such as $f$-divergence \cite{hu2013kullback, bayraksan2015data} and optimal transport distances \cite{kuhn2019wasserstein, gao2023distributionally} gained popularity as they enjoy better statistical properties.
A confidence set which will be of particular importance in this paper is the Kullback-Leibler confidence set
\(
  \mc S_{r}(\mb P_n) = \set{\mb Q\in \mc P}{\KL(\mb P_n, \mb Q)\leq r}
\)
where for all distributions $\mb P, \mb Q \in \mc P$ the quantity
\[
    \KL(\mb P, \mb Q) = \begin{cases}
      \int \log\left(\frac{\d \mb P}{\d \mb Q}(\xi)\right) \d \mb P(\xi) & {\rm{if}} ~\mb P\ll \mb Q,\\
      +\infty & {\rm{otherwise}}
    \end{cases}                         
  \]
denotes the Kullback-Leibler divergence.
\cite{vanparys2021data} point out that when the historical data points are samples from the unknown distribution $\Ptrue$, in the absence of any noise or corruption, such confidence set is ``minimal'' and its associated robust optimization formulation is consequently least conservative in a subtle sense which we make precise in Section \ref{sec:parametric-estimation}.

However, the desirable statistical properties of the Kullback-Leibler confidence set estimator evaporate in the presence of a distributional shift in which the data points $\{\xi_1, \dots, \xi_n\}$ are not sampled from $\Ptrue$ precisely but instead from a noisy \citep{vanparys2024efficient} or corrupted \citep{bennouna2022holistic} distribution instead.
In this work we will assume that our adversary can move a small fraction $\alpha \in [0, 1/2]$ of the mass of the distribution $\Ptrue$ to an arbitrary location, then, data is sampled from the perturbed distribution $\Pcor$.
Such distribution shift can be characterized precisely using the total variation (TV) distance
\[
  \TV(\mb P, \mb Q) = \inf \set{\epsilon> 0}{\mb Q(E) \leq \mb P(E)+\epsilon,~\mb P(E) \leq \mb Q(E)+\epsilon\quad \forall E\subseteq \Xi}
\]
for all distributions $\mb P, \mb Q \in \mc P$.
Indeed, the set of all conceivable corrupted distributions $\Pcor$ is here precisely characterized as $\{ \mb Q \in \mc P \; : \; \TV(\mb Q, \Ptrue) \leq \alpha\}$ as can be seen from the classical  
\citet{strassen1965existence} representation
\(
  \TV(\mb Q, \mb P)= \textstyle\inf_{\gamma\in \Gamma(\mb Q, \mb P)}\int \one{\xi \neq \xi'} \d \gamma(\xi, \xi'),
\)
where $\Gamma(\mb Q, \mb P)$ denote the transport polytope of all distributions on support $\Xi \times \Xi$ with given marginals $\mb Q$ and $\mb P$. This implies that $\TV(\mb Q, \mb P) \leq \alpha$ if and only if there exists a transport map that transforms distribution $\Pb$ into $\Qb$ by transporting less than a fraction $\alpha$ of the total probability mass. Other distortion distances could be considered as well. In fact, in the first three sections we only use that the total variation distance is convex, lower-semicontinuous and satisfies $\D(\mb P, \mb P)=0$ for all $\mb P\in \mc P$. However, at the moment we leave the generalization of our results to more general distortion distances as future work.

In settings with distribution shift, \cite{bennouna2022holistic} recently suggested confidence sets of the general form
\[
  \mc S_{r,\alpha}(\mb P_n) =  \set{\mb P \in \mc P}{\exists \mb Q\in \mc P~{\rm{\st}}~\KL(\mb P_n, \mb Q)\leq r, ~\D(\mb Q, \mb P)\leq \alpha},
\]
where the distortion distance $\D$ captures the distribution shift from the unknown distribution $\Ptrue$ of interest to a corrupted distributing $\Qb$, and the $\KL$ divergence accounts for an additional shift due to statistical error.
\cite{bennouna2022holistic} show that the associated robust optimization formulation \eqref{eq:dro} is least conservative in the same sense as the Kullback-Leibler confidence set estimator discussed earlier. We will indicate in the sequel that the associated ambiguity set is indeed ``minimal'' in the same subtle sense (which we have yet to make precise) as the KL ambiguity set in the absence of corruption.

Despite the alleged claim that the confidence set $\mc S_{r,\alpha}(\mb P_n)$ is ``minimal'' a careful observer may remark that it is nevertheless quite ``large'' and perhaps even ``too large''.
Indeed, it is easy to see that for any event $\bar{\xi} \in \Xi$,
the ambiguity set $\mc S_{r,\alpha}(\mb P_n)$ contains any distribution of the form $\alpha \delta_{\bar{\xi}} + (1-\alpha)\mb P_n$ where $\delta_{\bar{\xi}}$ is a degenerate point mass distribution at any arbitrary point $\bar{\xi}$ in our event set $\Xi$.
The previous observation implies that we have the lower bound
\[
\max_{\mb Q\in \mc S_{r,\alpha}(\mb P_n)}\E{\mb Q}{\ell(x, \xi)}\geq  (1-\alpha)\E{\mb P_n}{\ell(x, \xi)} + \alpha \textstyle\sup_{\bar \xi\in \Xi} \ell(x, \bar \xi).
\] 
This implies that with any distribution shift (however small $\alpha>0$) in the context of an unbounded loss function, even the alleged ``minimal'' confidence set yields a trivial cost estimate $+\infty$.
In the face of the previous observation, it would be tempting to conclude that the considered adversary may simply be too powerful to allow for a nontrivial insight in problems where the loss is not bounded. However, it can be remarked that the previous pathology occurs even in the absence of an adversary. Indeed, remark that it is straightforward to verify that
the ambiguity set $\mc S_{r,\alpha}(\mb P_n)$ contains the distribution $(1-\exp(-r)) \delta_{\bar{\xi}} + \exp(-r)\mb P_n$ as well for any $\alpha\geq 0$ and hence in particular also for $\alpha=0$ associated with the absence of corruption. Hence, no nontrivial formulations are possible even in the case of $\alpha = 0$ and $r>0$, i.e., with no corruption and merely statistical error, if our loss function is unbounded.

Perhaps more worryingly, \cite{bennouna2022holistic} show that even when the loss function is bounded, the associated ``least conservative'' robust formulation 
enjoys the saddle point condition
\begin{align*}
 \min_{x\in \cX} \max_{\mb Q\in \mc S_{r,\alpha}(\mb P_n)}\E{\mb Q}{\ell(x, \xi)} =& 
       \min_{x\in \cX} \sum_{i=1}^n q^\star_i \ell(x, \xi_i)  + q^\star_0 \max_{\bar \xi\in \Xi}\ell(x,\bar \xi)
\end{align*}
with $q^\star \in C^n_{r, \alpha}$ and $C^n_{r, \alpha}$ a convex set containing $(0, \tfrac 1n, \dots, \tfrac 1n)$. In particular, the associated robust formulation can be interpreted as an empirical risk minimization formulation where the data samples associated with a large loss tend to be emphasized rather than suppressed. This goes against the adage in robust statistics \citep{huber1981robust} of learning from corrupted data by attempting to identify and subsequently remove or at least suppress those samples the adversary has added to the training data. Indeed, it can be  remarked that even modern seminal work on robust learning such by \cite{diakonikolas2023algorithmic} is still based roughly on the idea of identifying and suppressing maliciously added outliers. Hence, there seems to be an interesting dichotomy between how robust optimization and robust statistics deals with the problem of data corruption and distribution shifts.

To bridge this seeming dichotomy between robust optimization and robust statistics, \citet{jiang2021dfo} recently proposed an optimistic formulation, rather than the worst-case formulation discussed here \eqref{eq:dro}, of the form
\begin{equation}
  \label{eq:dfo}
  \min_{x\in \cX} \min_{\mb Q\in \mc S(\mb P_n)}\E{\mb Q}{\ell(x, \xi)}
\end{equation}
which they argue can recover many classical robust statistical frameworks such as Winsorized regression and the Huber-skip estimator. Intuitively, whereas a worst-case perspective tends to emphasize large loss scenarios the previous optimistic formulation tends to suppress large loss scenarios more in line with what classical robust approaches seem to suggest. Unfortunately, formulation (\ref{eq:dfo}) is not a robust formulation in the classical sense---it is in fact antithetical to a robust formulation---and  hence any desirable effect it has protecting its decision against adversarial corruption is argued indirectly by comparison to well known robust statistical procedures rather than argued \textit{ex ante} from first principles. 
In the context of learning from corrupted data, resolving the described apparent conceptual dichotomy between the least conservative distributionally robust formulations and classical robust statistics from first principles will be the main research question of this paper.

\subsection{Contributions}

We will come to argue that the main cause of this dichotomy can be traced back to the lack of any structure imposed on the unknown distribution $\mb P^\star$. Indeed the distribution $\mb P^\star$ in the DRO literature is typically left completely arbitrary as in practice any structure imposed may be difficult to justify. However this lack of structure excludes the possibility to distinguish an artificial outlier added by the adversary from a genuine data sample. The removal of outliers becomes a possibility only when we impose that the data-generating distribution $\mb P^\star$ belongs to a subset of distributions which do not naturally produce outliers merely by sampling.
For instance, \cite{diakonikolas2023algorithmic} imposes the unknown distribution to be normal with unknown mean and variance.
If we know indeed that the distribution $\mb P^\star$ belongs to a parametric family $\mathcal{P}_\Theta=\set{\mathbb{P}_\theta}{\theta\in \Theta}\subset \mc P$, then the pathological distributions encountered earlier of the form $\alpha \delta_{\bar{\xi}} + (1-\alpha)\mb P_n$ or $(1-\exp(-r)) \delta_{\bar{\xi}} + \exp(-r)\mb P_n$ can often be dismissed out of hand and do not cause any further issues.
Similarly, we will argue that when leveraging the parametric structure in constructing ambiguity sets, the undesirable tendency of DRO formulations to emphasize large-loss scenarios disappears. 

More precisely, we study the problem of constructing confidence sets on the unknown distribution from data affected by corruption in its full generality. We introduce rigorously the problem in \cref{sec:parametric-estimation}. We seek \textit{minimal} confidence sets: that is the smallest possible sets that contains the unknown distribution with given desirable coverage guarantees. We prove that the (KL-TV) DRO ambiguity set, introduced by \cite{bennouna2022holistic}, is minimal in a strong sense---uniformly in the observed data.
We demonstrate that this result holds both
in the non-parametric case, and also the parametric case by constraining the ambiguity set to the parametric family. In particular, we show that the remarked conservatism of DRO approaches is an undesirable consequence of the adopted non-parametric perspective.
In \cref{sec:location-estimation} we study log-concave location parametric families often encountered in robust statistics.
We show that the DRO ambiguity set is never larger in terms of diameter than a confidence interval based on the classical estimator proposed by \cite{huber1964robust}. We do point out that such Huber confidence set is however \textit{worst-case} minimal by constructing an in which both these confidence sets are equal in terms of diameter. In \cref{sec:eventual-efficiency} we reveal that depending on the imposed coverage guarantees and corruption level this worst-case event may be a mirage and not not occur almost surely in which case a DRO ambiguity set is to be strictly preferred.

\subsection{Notation}

Given a compact set $S$ in a space equipped with norm $\| \cdot \|$ we will denote its Chebyshev center $\cheb(S)$ and radius $\radius(S)$ as the minimizer  and minimum in
\(
\min_{s\in S} \max_{s'\in S} \norm{s-s'},
\)
respectively. Given a set $K$ we denote its $\delta$-blowup with $K^\delta = \set{x}{\norm{y-x}\leq \delta, ~y\in K}$. 

\section{Confidence Set Estimation}
\label{sec:parametric-estimation}

We study the problem of estimating the unknown true distribution $\Ptrue$ from observed corrupted samples $\{\xi_1, \dots, \xi_n\}$.
Our objective is to construct a ``minimal'' confidence set that satisfies desired statistical coverage properties.

To study the problem in its full generality, we assume that the true distribution $\Ptrue = \Pb_{\thetatrue}$ belongs to a parametric family $\mathcal{P}_\Theta=\set{\mathbb{P}_\theta}{\theta\in \Theta}\subset \mc P$. The set  $\cP_\Theta$ imposes a structure on the distributions that may explain the observed data. 
Intuitively, learning from corrupted data should become easier once  structure is imposed on the data generating distribution as this structure may allow to distinguish gross outliers from genuine data points.
As the parametric set $\Theta$ is itself arbitrary $\cP_\Theta$ is quite flexible and ranges from completely non-parametric to the more restricted parametric families we discuss next.

\begin{example}[Non-Parametric]
    Let $\cP_\theta = \cP$, with $\Theta = \cP$, be the set of all distributions. This case corresponds to non-parametric estimation where no structure is assumed in the problem.
\end{example}

\begin{example}[Exponential Family]
  Let $T:\Xi\to \Re^d$ be a continuous function and $\mb P_0$ be a reference distribution. For each $\theta \in \Theta := \Re^d$, let $\mb P_\theta$ be such that
  \[
    \frac{\d \mb P_\theta}{\d \mb P_0} (\xi) = \exp(\theta\tpose T(\xi) -  A(\theta))\quad \forall \xi\in\Xi
  \]
  with cumulant function $A(\theta) = \log \int \exp(\theta\tpose T(\xi))\d \mb P_0$.
\end{example}

\begin{example}[Location Family]
  \label{ex:location}
  Let $\mb P_\theta$ and $\theta = \mu\in \Theta := \Re^d$ be such that
  \[
    \mb P_{\theta}(E) = \mb P_0(E-\mu) \quad \forall E\subseteq \Re^{d}
  \]
  where $\mb P_0$ is a reference distribution.
\end{example}

\begin{example}[Location-Scale Family]
  \label{ex:location-scale}
  Let $\mb P_\theta$ and $\theta = (\mu, \sigma)\in \Theta := \Re^d\times \Re_{++}$ be such that
  \[
    \mb P_{\theta}(E) = \mb P_0(\tfrac{(E-\mu)}{\sigma}) \quad \forall E\subseteq \Re^{d}
  \]
  where $\mb P_0$ is a reference distribution.
\end{example}

\subsection{Confidence Set Estimators and Coverage Guarantees}

To formalize our problem, we start by formally defining a confidence set estimator. We assume here that the order of the samples carries no statistical information and therefore we study estimators which are functions of the empirical distribution $\mb P_n$ rather than the data itself.

\begin{definition}[Confidence Set Estimators]
  \label{def: confidenceSet}
  A function $S:  \mathcal{P} \rightarrow \{ U \; : \; U \subseteq \Theta\}$ is called a confidence set estimator. A confidence set estimator $S$ is regular if the map $\mb P \mapsto S(\mb P)$ is continuous and closed-valued.
\end{definition}

A confidence set estimator is a set-valued mapping between elements in $\cP$ to subsets of parameters in $\Theta$.
A set estimator is desirable if it fences in the parameter of the true distribution $\theta^\star$ with ``high probability''. That is the probability that the event $\theta^\star \in S(\mb P_n)$ occurs is sufficiently large, where $\mb P_n$ denotes here the empirical distribution of the observed \textit{corrupted} data. In the non-parametric case, this is equivalent to $\Pb^\star \in S(\mb P_n)$. Formally, we say a confidence set estimator verifies a coverage guarantee with \resolution $r\geq 0$ and corruption level $\alpha \in [0,1]$ if
\begin{subequations}
 \label{eq:feasibility:parametric}
 \begin{align}
   \forall \thetatrue \in \Theta,~ \forall \Pcor \in \mathcal{P},~\D(\Pcor, \mb P_\thetatrue)\leq \alpha :&~ \limsup_{n \to \infty} \frac{1}{n} \log \Prob_{\Pcor}(\thetatrue \notin S(\mathbb{P}_n)) < -r.
   \label{eq:feasibility:parametric:r}\\
   \intertext{
That is, the probability of the confidence set fencing in the unknown parameter $\thetatrue$, for any corruption limited by $\alpha$, tends to one faster than $1-e^{-rn + o(n)}$.
We also treat the case  $r=0$ where we desire an asymptotic coverage guarantee with no specific convergence rate but where we have instead
}
\forall \thetatrue \in \Theta,~ \forall \Pcor \in \mathcal{P},~\D(\Pcor, \mb P_\thetatrue)\leq \alpha :&~ \limsup_{n \to \infty} \Prob_{\Pcor}(\thetatrue \notin S(\mathbb{P}_n)) =0. \label{eq:feasibility:parametric:0}
\end{align}
\end{subequations}

Notice that the guarantee \eqref{eq:feasibility:parametric} is required to hold \textit{for all distributions $\Pb_{\theta^\star}$ within the parametric family $\cP_\Theta$}. That is, the set estimator may use knowledge of (i) the data through $\Pb_n$, (ii) the fact that the data was generated from as samples from corrupted version of the unknown distribution and (iii) that the unknown distribution belongs to the parametric set $\cP_\Theta$. The last point is crucial as it will come to define what estimators verify such guarantees. The non-parametric case, $\cP_\theta = \cP$, imposes therefore rather strong requirements on our confidence set estimator which hold for all distributions $\Ptrue \in \cP$, while more restricted sets $\cP_\Theta$ are easier to comply with.
The statistical resolution $r$ and corruption level $\alpha$ are determine the strength of the desired guarantee imposed on our confidence set estimator. Here, the parameter $\alpha$ characterizes the resilience of our estimator to corruption, while $r$ sets the statistical power of its guarantee.  

We will focus in the sequel on confidence set estimators which are regular. 
Perhaps the most common class of confidence sets estimates are those of the form
\(
\hat{\mb P}\mapsto \tset{\theta\in\Theta}{\tnorm{\theta-\hat \theta(\hat{\mb P})}\leq \Delta(\hat {\mb P})}
\)
based on a parametric estimator $\hat{\theta} :\mc P\to\Theta$ and error function $\Delta:\mc P\to\Re_{++}$. In this setting (which we will encounter in Section \ref{sec:location-estimation}) regularity reduces to the continuity of the estimator and error function.
Finally, regularity can also quite naturally be interpreted as ensuring that the considered estimators are not overly sensitive to small changes to the observed empirical distribution and procedures enjoying such properties were studied already by \cite{mosteller1977data} who denoted such procedures as ``resistant''.

As pointed out before, coverage guarantees on a confidence estimator such as those stated in Equation \eqref{eq:feasibility:parametric} are a common occurrence in the robust optimization literature as they are sufficient to guarantee the out-of-sample performance of any downstream robust decision in (\ref{eq:dro}) based on the ambiguity set $S(\mathbb{P}_n)$.
As the trivial confidence set estimator $S(\mb P_n)=\Theta$ clearly points out, it is desirable to consider estimators which estimate ``small'' confidence sets. In particular, if two set estimators $S_1, S_2$ verify the coverage guarantee, and for all $\hat {\mb P} \in \mc P$, 
$S_1(\hat {\mb P}) \subseteq S_2(\hat{\mb P})$, then clearly $S_1$ is preferred over $S_2$ as the first estimator manages to fence in the unknown distribution within a smaller set. Given this observation it is now quite natural to try and find a ``minimal'' set estimator.

\subsection{A Uniformly Minimal Confidence Set}

Let us first build intuition on the set estimator we will introduce and which we allege is minimal. If $\mb P_n$ is an empirical distribution sampled from $\Pcor$, large deviation theory \citep{dembo2009large} ensures that $\KL(\mb P_n,\Pcor) \leq r$ with probability at least $e^{-rn+o(n)}$, when the support $\Xi$ is finite. Hence, if $\Pcor$ is a corrupted distribution verifying $\D(\Pcor,\Ptrue) \leq \alpha$, then the confidence set $\{\mb P \in \mc P \; : \; \exists \mb Q\in \mc P~\st~\KL(\mb P_n, \mb Q)\leq r,~\D(\mb Q, \mb P)\leq \alpha \}$ should contain the true distribution $\Ptrue$ with probability at least $1-e^{-rn+o(n)}$. Now, as the true distribution is known to belong to the parametric family $\cP_\Theta$, we can restrict the set to the family. It is therefore natural to introduce the (KL-TV) DRO set estimator
\begin{align}
  \label{eq:KL-confidence-estimator}
  \SpKL{r}{\alpha}: \hat{\mb P} \mapsto \set{\theta\in \Theta}{\exists \mb Q\in \mc P, ~\KL(\hat{\mb P}, \mb Q)\leq r, ~\D(\mb Q, \mb P_{\theta})\leq \alpha}.
\end{align}

Perhaps the first immediate remark one can make is that the DRO confidence set estimator is not necessarily regular; a point to which we will come back to later on. 
For now we associate with this ambiguity set the statistical resolution function
\begin{equation}
  \label{eq:resolution_function_set}
  r^\alpha(\hat {\mb P}, \theta) = \min \tset{\KL(\hat {\mb P}, \mb Q)}{\D(\mb Q, \mb P_\theta)\leq \alpha}, \quad \forall \theta \in \Theta, \; \forall \hat{\Pb} \in \cP,
\end{equation}
and remark indeed that we can write the ambiguity sets as its sublevel set, i.e., $\SpKL{r}{\alpha}(\hat {\mb P}) = \tset{\theta\in \Theta}{r^\alpha(\hat {\mb P}, \theta)\leq r}$ for any $r\geq 0$. Intuitively, $r^\alpha(\hat {\mb P}, \theta)$ returns the minimum radius so that the estimate $\theta$ is in the KL-TV ambiguity set centred around $\hat {\mb P}$, that is $\theta \in \SpKL{r^\alpha(\hat {\mb P}, \theta)}{\alpha}(\hat {\mb P})$.
As we indicate in the following example, the statistical resolution function provides considerable insights into the proposed confidence set estimator. In particular, the example shows that while the KL-TV confidence set $\SpKL{r}{\alpha}(\hat {\mb P})$ is convex in the non-parametric case $\cP_{\Theta} = \cP$, it can be non-convex in the parametric case.

\begin{example}[Normal Family]
  \label{ex:normal-family}
  Consider the normal location family $\mc P_\Theta = \set{N(\theta,1)}{\theta\in \Re}$ parametrized in the unknown mean $\theta\in\Re$ (see also Example \ref{ex:location}). Let's take $\theta^\star=0$ here and we will consider
  \( 
  \hat{\mb P} = (1-\alpha)N(\theta^\star, 1) + \alpha N(\theta^\star+\Delta, 1)
  \)
  for some $\Delta\gg 1$. 
  In Figure \ref{fig:resolution-example} we depict the function $\theta\mapsto r^\alpha(\hat {\mb P}, \theta)$ for $\Delta=5$ and $\alpha=0.3$.
  It can be remarked that $r^\alpha(\hat {\mb P}, \theta^\star) = 0$ and hence our confidence set estimator fences in the unknown parameter ($\theta^\star\in \SpKL{r}{\alpha}(\hat{\mb P})$) for any $r\geq 0$. It should also not come as a surprise that in the context of noise ($\alpha>0$) we have that the ambiguity set $\SpKL{0}{\alpha}(\hat{\mb P}) \neq \{ \theta^\star \} $ does not reduce to a singleton. Intuitively, due to the presence of noise the unknown parameter $\theta^\star$ can not be recovered exactly even with access to infinite data but rather any $\theta$ in  $\SpKL{0}{\alpha}(\hat{\mb P})$ (indicated as the red region in Figure \ref{fig:resolution-example}) is compatible with the observed distribution $\hat {\mb P}$. 
\end{example}

\begin{figure}[h!]
  \centering
  \includegraphics[height=6cm]{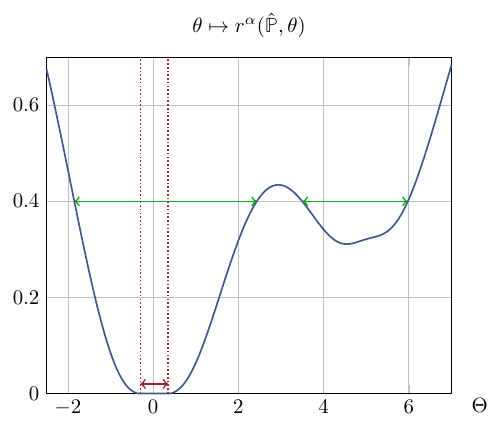}
  \caption{The statistical resolution function can be nonconvex resulting in nonconvex set estimators $\SpKL{r}{\alpha}(\hat {\mb P})$. For instance, with $\alpha = 0.3$, the set $\SpKL{0.4}{\alpha}(\hat{\mb P})$ ($= \{ \theta \; : \; r^\alpha(\hat{\Pb},\theta) \leq 0.4 \}$) indicated in green is the union of two disjoint intervals. The region in red denotes $\SpKL{0}{\alpha}(\hat{\mb P})$ identified as the roots of the statistical resolution function. The fact that the latter region is not a singleton indicates that learning the unknown parameter with corruption exactly is not possible, even with infinite data.}
  \label{fig:resolution-example}
\end{figure}

The following result shows that the KL-TV confidence set is \textit{uniformly minimal}: any regular confidence set estimator which satisfies the coverage guarantee (\ref{eq:feasibility:parametric}) must contain the KL-TV confidence set whatever the observed distribution.

\begin{theorem}[KL-TV Uniform Minimality]
  \label{thm: KLefficiency:noise}
  For any regular confidence set estimator $S$ verifying the coverage guarantee \eqref{eq:feasibility:parametric} with $r\geq 0$ we have $\SpKL{r}{\alpha}(\hat{\mb P}) \subseteq S(\hat{\mb P})$ for all $\hat{\mb P} \in \mc P$. 
\end{theorem}
\begin{appendixproof}[Proof of \cref{thm: KLefficiency:noise}]
  We prove this results for $r>0$ and $r=0$ separately.\\
  \textbf{Case $r>0$}:
  Let's assume for the sake of contradiction that there exists $\hat{\mb P}\in \mc P$ and $\theta\in\Theta$ so that $\theta\in \SpKL{r}{\alpha}(\hat{\mb P})$ but $\theta\not\in S(\hat{\mb P})$. By definition of $\theta\in\SpKL{r}{\alpha}(\hat{\mb P})$ we have that there exists $\Pb^c\in \mc P$ so that $\KL(\hat{\mb P}, \Pb^c)\leq r$ and $\D(\Pb^c, \mb P_{\theta})\leq \alpha$.
  Let
  \(
  \mc D(\theta) = \set{\mb P'}{\theta\not \in S(\mb P')}
  \)
  denote the set of all distributions so that the confidence set estimator $S$ does not contain $\theta$. By construction we have that $\hat{\mb P} \in \mc D(\theta)$ and the set $\mc D(\theta)$ is an open set as the confidence set estimator $\mc S$ is regular. Hence, a standard result of Sanov's large deviations lower bound \citep{dembo2009large} is that
  \begin{align*}
  \liminf_{n \to \infty} \frac{1}{n} \log \Prob_{\Pb^c}(\theta \notin S(\mathbb{P}_n))
  &=
  \liminf_{n \to \infty} \frac{1}{n} \log \Prob_{\Pb^c}(\mathbb{P}_n \in \cD(\theta))\\
  &\geq
  - \inf_{\mb P' \in \interior \mc D(\theta)} \KL(\mb P', \Pb^c) \geq -\KL(\hat{\mb P}, \Pb^c) \geq r
   \end{align*}
  directly contradicting the feasibility of the confidence set estimator $S$ in Equation \eqref{eq:feasibility:parametric:r} as $\D(\Pb^c, \mb P_{\theta})\leq \alpha$.

  \textbf{Case $r=0$:} Assume that there exists $\Pb^c\in \mc P$ and $\theta \in \Theta$ with $\theta \in \SpKL{0}{\alpha}(\Pb^c)$ but $\theta \not\in S(\Pb^c)$. From regularity of $S$ we have hence that $\mc D(\theta)=\tset{\hat {\mb P}}{\theta \not\in S(\hat {\mb P})}$ is an open set containing $\Pb^c$.
  Consequently, there exists $\epsilon>0$ so that $\set{\mb P'}{\LP(\mb P', \Pb^c)< \epsilon} \subseteq \mc D(\theta)$. Hence, we have from \cite{varadarajan1958convergence} that
  \[
    \liminf_{n \to \infty} \Prob_{\Pb^c}(\theta \notin S(\mathbb{P}_n)) \geq \liminf_{n \to \infty} \Prob_{\Pb^c}(\mb P_n \in \mc D(\theta)) \geq\liminf_{n \to \infty} \Prob_{\mb P^c}(\LP(\mb P_n, \Pb^c) < \epsilon) = 1
  \]
  where $\TV(\Pb^c, \mb P_\theta)\leq \alpha$ as $\theta \in \SpKL{0}{\alpha}(\Pb^c)$ directly contradicting the feasibility of the confidence set estimator $S$ in Equation \eqref{eq:feasibility:parametric:0}.
\end{appendixproof}

\cref{thm: KLefficiency:noise} is a rather strong result as it holds uniformly for any observed distribution.  
However, the set estimator $\SpKL{r}{\alpha}$ may itself not enjoy the coverage guarantee (\ref{eq:feasibility:parametric}) when our event set $\Xi$ is not discrete.
Indeed, even in a case without corruption ($\alpha=0$), we have that $\SpKL{r}{0}(\mb P_n)\subseteq \tset{\mb Q\in \mc P_\Theta}{\mb P_n\ll \mb Q}$ and hence the confidence set estimate $\SpKL{r}{0}(\mb P_n)$ fails to contain the true distribution if family $\cP_\Theta$ contains merely continuous distributions, as $\mb P_n$ is always a discrete measure supported on at most $n$ points.
Informally, hence it would appear that the set estimator $\SpKL{r}{\alpha}$ is simply too ``small'' to be useful outside of the case where $\Xi$ is discrete. This issue can be directly attributed to the lack of regularity of the KL distance when  $\Xi$ is continuous as pointed out by \citet{yang2018robust}.
However, as pointed out by \cite{dembo2009large, yang2018robust} and more recently by \cite{liu2023smoothed}, out-of-sample feasibility can be salvaged by considering a smoothed distance instead. For instance, let us define a smooth KL distance $\KL^\delta(\hat{\mb {P}}, \mb Q) = \inf_{\LP(\hat{\mb {P}}, \mb P')\leq \delta} \KL(\mb P', \mb Q)$, where $\LP$ is the L\'evy-Prokhorov metric \citep{prokhorov1956convergence}, and consider instead the confidence set estimator
\[
  \SpKL{r}{\alpha}^\delta: \hat{\mb P} \mapsto \tset{\theta\in \Theta}{\exists \mb Q\in\mc P~\st~\KL^\delta(\hat{\mb P}, \mb Q)\leq r,~\D(\mb Q, \mb P_\theta)\leq \alpha}.
\]

\begin{theorem}[Coverage Guarantee]
  \label{thm:feasibility:2}
  For any $\delta>0$, $\SpKL{r}{\alpha}^\delta$ satisfies the coverage guarantee \eqref{eq:feasibility:parametric} with \resolution $r\geq 0$.
\end{theorem}
\begin{appendixproof}[Proof of \cref{thm:feasibility:2}]
  We prove this results for $r>0$ and $r=0$ separately.\\
  \textbf{Case $r>0$}:
  Fix a $\theta \in \Theta$, $\Pb^c \in \mathcal{P}$ with $\D(\Pb^c, \mb P_\theta)\leq \alpha$.
  We remark that by construction of the confidence set estimator $\SpKL{r}{\alpha}^\delta$ we have that $\theta \notin \SpKL{r}{\alpha}^\delta(\mathbb{P}_n)$ implies the event $\mb P_n \in \mc C = \tset{\mb P''\in \mc P}{\KL^\delta(\mathbb{P}'', \Pb^c)> r}$ must have occurred.
  Sanov's theorem guarantees that
  \begin{align*}
    \limsup_{n \to \infty} \frac{1}{n} \log \Prob_{\Pb^c}(\theta \notin  \SpKL{r}{\alpha}^\delta(\mathbb{P}_n)) \leq& \limsup_{n \to \infty} \frac{1}{n} \log \Prob_{\Pb^c}(\mb P_n \in \mc C)\\
    \leq & - \inf \set{\KL(\mathbb{P}'', \Pb^c)}{\mb P''\in \cl(\mc C)}.
  \end{align*}
  It remains to show that $\inf \set{\KL(\mathbb{P}'', \Pb^c)}{\mb P''\in \cl(\mc C)}>r$. We first remark that as the set $\cl(\mc C)$ is closed and the sublevel sets $\set{\mb P''}{\KL(\mb P'', \mb P^c)\leq r}$ are compact 
  the infimum is achieved at some $\mb P''^\star$. Hence, we have $\inf \set{\KL(\mathbb{P}'', \Pb^c)}{\mb P''\in \cl(\mc C)} = \KL(\mathbb{P}''^\star, \Pb^c)$. Suppose for the sake of contradiction that $\KL(\mathbb{P}''^\star, \Pb^c)\leq r$ than there exists $\mb P''_k\in \mc C$, and hence $\inf_{\LP(\mb P''_k, \mb P')\leq \delta} \KL(\mathbb{P}', \Pb^c)>r$, with limit $\mb P''^\star$.
  However, for $k$ large enough so that $\LP(\mb P''_k, \mb P''^\star)\leq \delta$ this would imply with $\KL(\mb P''^\star, \Pb^c)\leq r$ that $\KL^\delta(\mb P''_k, \Pb^c)\leq r$; a contradiction.

  \textbf{Case $r=0$.}  Fix a $\theta \in \Theta$, $\Pb^c \in \mathcal{P}$ with $\D(\Pb^c, \mb P_\theta)\leq \alpha$. We remark that by construction of the confidence set estimator $\SpKL{0}{\alpha}^\delta$ we have that $\theta \notin  \SpKL{0}{\alpha}^\delta(\mb P_n)$ implies the event $\mb P_n \in \mc C \defn \tset{\mb P''\in \mc P}{\KL^\delta(\mathbb{P}'', \Pb^c)> 0} \subseteq \set{\mb P''\in \mc P}{\LP(\mb P'', \Pb^c) > \delta}$ has occurred. Hence, from \cite{varadarajan1958convergence} we have
  \[
   \limsup_{n\to\infty} \Prob_{\Pb^c}(\theta \notin  \SpKL{0}{\alpha}^\delta(\mb P_n))  \leq \limsup_{n\to\infty} \Prob_{\Pb^c}(\mb P_n \in \mc C) \leq \limsup_{n\to\infty}\Prob_{\Pb^c}(\LP(\mb P_n, \Pb^c) > \delta) =0
 \]
 establishing the claim.
\end{appendixproof}

The necessity for this slight blowup was already remarked by \cite{dembo2009large} in the context of universal hypothesis testing. We observed that $\SpKL{r}{\alpha} (\hPb) \subseteq \SpKL{r}{\alpha}^\delta(\hPb) \subset \SpKL{r}{\alpha}^{\delta'}(\hPb)$ for any $\delta' >\delta>0$ and $\hPb \in \mc P$---in particular, $\SpKL{r}{\alpha}=\SpKL{r}{\alpha}^0$. As we now indicate the additional blowup is negligible as the estimator $\SpKL{r}{\alpha}$ can be recovered as the limiting case $\delta\downarrow 0$.

\begin{theorem}[Limit]
  \label{thm:approximation:B}
  We have
  $\lim_{\delta\downarrow 0}\SpKL{r}{\alpha}^\delta( \hat{\mb P}) \defn \cap_{\delta > 0} \SpKL{r}{\alpha}^\delta( \hat{\mb P}) = \SpKL{r}{\alpha}(\hat{\mb P})$ for all $\hat{\mb P}\in\mc P$.
\end{theorem}  
\begin{appendixproof}[Proof of \cref{thm:approximation:B}]
  Consider any $\theta \in \cap_{\delta > 0} \SpKL{r}{\alpha}^\delta( \mb P')$. That is, let $\delta_k>0$  tend to zero and then we have that $\KL(\mb P^{\delta_k}, \mb Q^{\delta_k})\leq r$ and $\TV(\mb Q^{\delta_k}, \mb P_\theta)\leq \alpha$ for some $\mb Q^{\delta_k}\in \mc P$ and $\mb P^{\delta_k}\in \mc P$ with $\LP(\mb P', \mb P^{\delta_k})\leq \delta_k$. The last inequality implies that $\mb P^{\delta_k}$ converges to $\mb P'$. We will show that $\mb Q^{\delta_k}$ has a subsequential limit $\mb Q\in \mc P$.
  As the Kullback-Leibler divergence and total variation distance are lower semicontinuous we have
  \[
    \KL(\mb P', \mb Q) \leq \liminf_{k\to\infty}\KL(\mb P^{\delta_k}, \mb Q^{\delta_k}) \leq r, \quad \TV(\mb Q, \mb P_\theta) \leq \liminf_{k\to\infty}\D(\mb Q^{\delta_k}, \mb P_\theta) \leq \alpha
  \]
  and hence $\theta\in \SpKL{r}{\alpha}(\mb P')$.
  
  We now show that $\mb Q^{\delta_k}$ has indeed a subsequential limit $\mb Q\in \mc P$. It suffices to show that $\mc Q = \cl(\cup_{k\geq 0} \mb Q^{\delta_k})$ is compact. By \cite{prokhorov1956convergence} this is equivalent to showing that $\mc Q$ is tight. Given $\epsilon>0$, we show how to construct $K$ so that for all $\mb Q\in \mc Q$ we have $\mb Q[\Xi\setminus K] < \epsilon$.
  Let us define $\epsilon'>0$ sufficiently small so that
  \begin{equation}
    \label{eq:sufficiently-small-epsilon}
    \log\left(\frac{\epsilon'}{\epsilon}\right)\epsilon' + \log\left(\frac{1-\epsilon'}{1-\epsilon}\right)(1-\epsilon')>r.
  \end{equation}
  Let $K_1$ be such that we have
  \(
  \mb P'[\Xi\setminus K_1] < \epsilon'/2
  \)
  and $k_1$ so that for all $k\geq k_1$ we have $\delta_k<\epsilon'/2$.
  Define the compact set $K_1'=K_1^{\epsilon'/2}$.
  We have for all $k\geq k_1$ that
  \(
    \mb P^{\delta_k}[\Xi\setminus K'_1] - \mb P'[\Xi\setminus K_1] \leq \mb P^{\delta_k}[\Xi\setminus K^{\delta_k}_1] - \mb P'[\Xi\setminus K_1] \leq \LP(\mb P', \mb P^{\delta_k})\leq \delta_k
  \)
  and hence $\mb P^{\delta_k}[\Xi\setminus K'_1]< \epsilon'$. We also that
  \[
    \log\left(\frac{\mb P^{\delta_k}[K'_1]}{\mb Q^{\delta_k}[K'_1]}\right)\mb P^{\delta_k}[K'_1] + \log\left(\frac{\mb P^{\delta_k}[\Xi\setminus K'_1]}{\mb Q^{\delta_k}[\Xi\setminus K_1]}\right)\mb P^{\delta_k}[\Xi\setminus K'_1] \leq \KL(\mb P^{\delta_k}, \mb Q^{\delta_k})\leq r.
  \]
  and hence combined with $\mb P^{\delta_k}[\Xi\setminus K'_1]< \epsilon'$ and Equation (\ref{eq:sufficiently-small-epsilon}) it follows that we must have
  \(
  \mb Q^{\delta_k}[K'_1]<\epsilon
  \)
  for all $k\geq k_1$. Let now $K_2$ be such that for all $0\leq k < k_1$ it follows that $\mb P^{\delta_k}[\Xi\setminus K_2]< \epsilon'$. We then have for all $k\geq 0$ that
  \(
  \mb P^{\delta_k}[\Xi\setminus (K'_1 \cup K_2)]< \epsilon'.
  \)
\end{appendixproof}

The previous implies in a sense that the KL-TV confidence set $\SpKL{r}{\alpha}$ can more accurately be thought of as ``infimal'' rather than ``minimal'' in the problem of finding the smallest confidence set estimator verifying the coverage guarantee (\ref{eq:feasibility:parametric}). That is, it is smaller than any other regular confidence set estimator as pointed out in \cref{thm: KLefficiency:noise} and, although not verifying the guarantee itself, it is the natural limit within a family of slightly larger confidence set predictors ($\SpKL{r}{\alpha}^\delta$ for $\delta>0$) verifying the guarantee, as pointed out by \cref{thm:approximation:B}.

\subsection{On the Shape and Tractability of Minimal Confidence Sets}
The minimality of the KL-TV confidence set, by \cref{thm: KLefficiency:noise}, provides several interesting insights. 
In the non-parametric case---typically considered in DRO---the KL-TV confidence set is uniformly minimal despite being quite ``large'' given that it needs to contain several pathological distributions as we have pointed out in the introduction. The reason for this is that when no structure is known on the true distribution, any arbitrary distribution that is statistically consistent with our data must be included in the ambiguity set. This leads to several pathological phenomena in DRO. For example, outliers in the data are now the most relevant data points, as they strongly influence the loss, and as they are no less likely due to corruption than the rest of the data points can not be rejected. This explains why classical DRO ambiguity sets tend to emphasize outliers, rather than suppressing them as would be the norm in classical robust statistics.

In the next section, we will show that under parametric structure, these pathological cases vanish from the KL-TV confidence set. Furthermore, the estimator becomes in fact intimately related to classical robust statistics approaches, namely the Huber estimator. This implies that the remarked ``conservativeness'' of DRO approaches is not inherently related to the DRO perspective, but rather a natural requirement of the non-paramteric setting typically considered.

A natural question is now if the (parametric) KL-TV confidence set is uniformly minimal, then why would we use anything else than a KL-TV confidence (or its blowup)? While minimal, the KL-TV confidence set is generally not tractable under a parametric restriction. Indeed, while convex in the non-parametric case, it is no longer convex under a parametric restriction (see also Example \cref{ex:normal-family}).
As the KL-TV confidence set is proven to be minimal, we will refer to it in the next sections as the DRO confidence set, to distinguish it from other types of robust estimators.

\section{Location Estimation in Logconcave Families}
\label{sec:location-estimation}

In the previous section we studied general parametric families corrupted by an adversary characterized by the total variation distance.
In this section, we specialize our considered estimators to the classical setting of robust statistics in which the randomness is univariate $\xi \in \Xi = \Re$ and
$\theta \in \Theta=\Re$ is the location parameter in a univariate location family.
Furthermore, we let the distribution $\mathbb{P}_0$ characterizing our location family be an absolutely continuous log-concave generating distribution with density function $g:\Xi \to \Re_{++}$ symmetric around the origin. 
The parametric distribution $\mb P_\theta$ has density $g(\cdot - \theta)$. In other words, if $\Tilde{\xi} \sim \Pb_\theta$, then $\Tilde{\xi} - \theta \sim \Pb_0$.

\begin{example}[Location Families]\label{ex: location families}
  The following are particular instances of Example \ref{ex:location}.
    \begin{enumerate}
        \item Normal Distributions: $g: \xi \mapsto \frac{1}{\sigma\sqrt{2\pi}} e^{-\frac{1}{2}\left( \tfrac{\xi}{\sigma}\right)^2}$ with $\sigma > 0$. 
        \item Logistic Distributions: $g: \xi \mapsto \tfrac{e^{-\xi/\sigma}}{\left(\sigma\left( 1+ e^{-\tfrac{\xi}{\sigma}}\right)^2\right)}$ with $\sigma>0$. 
         \item Laplace Distribution: $g: \xi \mapsto \frac{1}{2\sigma} e^{\tfrac{\abs{\xi}}{\sigma}}$ with $\sigma>0$.
    \end{enumerate}
\end{example}

We will not discuss whether or not the parametric restriction discussed in this section is justifiable in practice.
We study such setting as it is central to the robust statistics litterature and in particular the seminal work of \cite{huber1964robust}. Our focus will be to show from first first principles, that in this natural setting, the minimal DRO confidence set $\SpKL{r}{\alpha}$ is intimately related to the classical robust estimators of \cite{huber1964robust}.

This parametric setting is a special case of the family introduced in Example \ref{ex:location} for which we have shown that the confidence set $\SpKL{r}{\alpha}$ is ``minimal'' in the sense discussed in previous sections.
However, this confidence set estimator is not practical and lacks tractability. A natural question is therefore then whether we can design in this setting a more practical confidence set that preserves, to the best possible, the statistical properties of the DRO confidence set. As the parameter set $\Theta \subset \Re$ here is one dimensional, a natural candidate for a simpler subclass of confidence set estimators is the class of confidence intervals.
Clearly, the DRO confidence set does not always belong to this class as it can be non-convex; see Example \ref{ex:normal-family} and Figure \ref{fig:resolution-example}.
The next section discusses the restriction of our problem of finding minimal confidence set estimators with given coverage guarantees to the case of confidence intervals.

\subsection{Confidence Intervals and Estimators}

A confidence interval can be identified through its center---an estimator---and radius. Let us first define estimators in this context.

\begin{definition}[Estimator]
  A function $E:\mc P\to \Theta$ is denoted an estimator. An estimator is denoted as regular if the mapping $E$ is continuous for the weak topology. 
\end{definition}

An estimator naturally defines a (fixed-width) confidence interval as $[E(\mb P_n) - \Delta, E(\mb P_n) + \Delta]$ for a given radius $\Delta>0$. Within this class of set estimators, the coverage guarantee \eqref{eq:feasibility:parametric} becomes the following condition on the estimator $E$ and the confidence interval radius $\Delta$
\begin{subequations}
 \label{eq:point-estimate:guarantee:noise}
 \begin{align}
   \forall \thetatrue \in \Theta,~ \forall \Pcor \in \mathcal{P},~\TV(\Pcor, \mb P_\thetatrue)\leq \alpha :&~ \limsup_{n \to \infty} \frac{1}{n} \log \Prob_{\Pcor}(|E(\mb P_n)- \thetatrue|> \Delta) < -r \label{eq:point-estimate:guarantee:noise:r}\\
   \intertext{for some statistical resolution $r>0$. 
The previous inequality guarantees that the fixed width confidence interval $[E(\mb P_n) - \Delta, E(\mb P_n) + \Delta]$ around the considered estimator fences in the true paramter $\thetatrue$ with probability increasing exponentially fast to one at speed exceeding $1-e^{-rn+o(n)}$. 
We will also discuss the special case $r=0$ in which we require merely}
   \forall \thetatrue \in \Theta,~ \forall \Pcor \in \mathcal{P},~\TV(\Pcor, \mb P_\thetatrue)\leq \alpha :&~\limsup_{n \to \infty}  \Prob_{\Pcor}(|E(\mb P_n)- \thetatrue|> \Delta) =0. \label{eq:point-estimate:guarantee:noise:0}
 \end{align}
\end{subequations}

For a given estimator $E$, we naturally seek the smallest associated confidence interval which verifies the coverage guarantee \eqref{eq:point-estimate:guarantee:noise} at a given statistical resolution $r$ and under corruption level $\alpha$.  We define the radius of this interval\footnote{Note that as the radius is defined as an infimum, the guarantee \eqref{eq:point-estimate:guarantee:noise} does not hold necessarily at $\Delta=\rad{E}{r}{\alpha}$ but rather only at all $\Delta > \rad{E}{r}{\alpha}$.} here as
$$
\rad{E}{r}{\alpha} = \inf\{\Delta \geq 0 \; : \; E,\Delta \; \rm{satisfy~the~coverage~ guarantee~ \eqref{eq:point-estimate:guarantee:noise}}\}.
$$
Clearly, with a given statistical resolution $r$ and corruption level $\alpha$, the estimator $E$ with smallest radius $\rad{E}{r}{\alpha}$ is preferred, and is precisely what we are looking for.

\subsection{Examples}

Before proceeding any further, let us present some examples of estimators and their associated confidence intervals. The most straightforward such estimator is the empirical mean defined as $E(\Pb_n) = \Eb_{\Pb_n}[\xi]$. For the paramteric family of normal distributions (Example \eqref{ex: location families}, 1), the Chernoff concentration bound gives
\[
\Prob_{\Pb_{\theta^\star}}\left( \abs{E(\Pb_n) - \theta^\star} > \Delta\right) \leq 2e^{-\frac{n\Delta^2}{2\sigma^2}},
\]
for all $\Delta\geq0$.
This implies, by chosing $\Delta=\sigma \sqrt{2r}$, that $\rad{E}{r}{0} \leq \sigma \sqrt{2r}$. Hence, the mean is a reasonable estimator in the absence of corruption, however, as already \citet{huber1981robust} remarks, it is inadequate in the presence of corruption ($\alpha>0$). In fact, moving any fraction of the data points can change the empirical mean arbitrarily, hence $\rad{E}{r}{\alpha} = \infty$ for any $\alpha>0$. Given that observation, a more sensible estimator under corruption is the median.

\begin{example}[Median]
  \label{exp:median}
  The median  $E^\median$ is an estimator verifying $\Eb_{\hat \Pb}\left[ \sign(\xi - E^\median(\hat \Pb)) \right] = 0$ for any $\hat \Pb \in \cP$.  
\end{example}
In \cref{sec:median}, we discuss in details the confidence interval induced by the median and will show in a precise sense that it reaches the smallest possible radius in the case of $r=0$ and $\alpha>0$.
Both these estimators belong to the general class of $\varphi-$estimators.

\begin{definition}[$\varphi-$Estimators]
    Let $\varphi:\Re\to\Re$ be an odd non-decreasing function. A $\varphi$-Estimator, denoted as $E_{\varphi}$, is any estimator verifying $\Eb_{\hat \Pb}\left[\varphi(\xi - E_\varphi(\hat {\mb P}))\right] = 0$ for all $\hat {\mb P} \in \mc P$.
\end{definition}
The mean is then a $\varphi$-estimator with $\varphi(\xi) = \xi$, for all $\xi$, and the median with $\varphi(\xi) = \sign(\xi)$.
We remark that $\varphi$-estimators are desirable as they are tractable. In fact, they can be characterized as the minimizer of a convex optimization. 

\citet[Theorem 2.6]{huber1981robust} shows that the particular class of $\varphi$-estimators is regular whenever the function $\varphi$ is bounded and strictly monotone. Clearly, the empirical mean is not regular. Close inspection show that in fact the median is also not regular as it is not uniquely defined for some distributions. It is however widely known that for an $\varphi$-estimator to be resilient against corruption its characterizing influence function $\varphi$ needs to be bounded \citep{huber1981robust} ruling out the mean as a sensible estimator in the context of data corruption. Despite its lack of regularity we show in \cref{prop:lower-bound-gross-error-margin-noise} that the median is resilient against corruption when strong statistical guarantees are not required---a fact unsurprising in the light of \cite{huber1981robust}.

In the context of designing confidence intervals in the robust statistics literature, a celebrated $\varphi-$estimator resilient against corruption is the Huber estimator \citep[Section 10.7]{huber1981robust}.

\begin{example}[Huber Estimator]\label{eg:huber-estimator}
    The Huber estimator $E_{\delta, k}^\huber$ is defined as the estimator associated with
    \begin{equation*}
        \phihuber_{\delta,k}(\xi) = 
        \begin{cases}
          -k \quad &\text{if} \; \log \left(\tfrac{g(\xi-\delta)}{g(\xi+\delta)}\right) \leq -k,\\
          \log \left(\tfrac{g(\xi-\delta)}{g(\xi+\delta)}\right) \quad &\text{if} \; -k\leq \log \left(\tfrac{g(\xi-\delta)}{g(\xi+\delta)}\right) \leq k, \\
          k \quad &\text{if} \; \log \left(\tfrac{g(\xi-\delta)}{g(\xi+\delta)}\right) \geq k        
        \end{cases}
    \end{equation*}
with parameters $\delta,k\geq 0$.
\end{example}

The Huber estimator was first defined in the seminar paper \cite[Theorem 1]{huber1964robust} in a slightly different form as $\varphi_k(\xi) = -\tfrac{g'(\xi)}{g(\xi)}$  if $\abs{\tfrac{g'(\xi)}{g(\xi)}} \leq k$; $k$ if $\tfrac{g'(\xi)}{g(\xi)} \geq k$; and $-k$ if $\tfrac{g'(\xi)}{g(\xi)} \leq -k$ in the context of designing estimators minimizing maximal asymptotic variance under corruption. This estimator can be recovered intuitively as the limit of $\phihuber_{\delta,k}/(2\delta)$ when $\delta \to 0$.
In the case of normal distributions discussed in Example \ref{ex: location families}(1) the Huber influence function  can be written as
    \begin{equation*}
        \phihuber_{\delta, k}(\xi) = 
        \begin{cases}
          -k \quad &\text{if} \; 2\delta \xi / \sigma^2 \leq -k,\\
          2\delta \xi / \sigma^2 \quad &\text{if} \; \abs{2\delta \xi / \sigma^2} \leq k, \\
          k \quad &\text{if} \; 2\delta \xi / \sigma^2 \geq k 
        \end{cases}
    \end{equation*}
for $\delta, k\geq 0$.

\subsection{A Fundamental Lower bound and a Worst-case Minimal Estimator}
\label{sec:fund-lower-bound}

In the search for minimal (fixed width) confidence intervals verifying the coverage guarantees \eqref{eq:point-estimate:guarantee:noise}, we will first prove a fundamental lowerbound on the radius of any such confidence interval. This lower bound happens to be the largest radius of the DRO confidence set estimator \eqref{eq:KL-confidence-estimator} denoted here as
\begin{equation}\label{eq: KL radius}
  \KLrad{r}{\alpha} \defn \textstyle\sup_{\hat{\mb P} \in \mc P}\radius(\SpKL{r}{\alpha}(\hat{\mb P})).
\end{equation}
This is not surprising given the results of the previous section, as we showed the DRO confidence set to be minimal for verifying the coverage guarantees.

\begin{theorem}\label{thm:efficiency:location}
  For any regular estimator $E$, we have $\rad{E}{r}{\alpha} \geq \KLrad{r}{\alpha}$ for all $r,\alpha \geq 0$.
\end{theorem}
\begin{appendixproof}[Proof of \cref{thm:efficiency:location}]
  For the sake of contradiction assume that $\rad{E}{r}{\alpha}<\KLrad{r}{\alpha}$ for some regular estimator $E$ for some $r\geq 0$.
  For any $\Delta'<\rad{E}{r}{\alpha}$, there hence must exist a distribution $\hat {\mb P}$ so that $a,b\in \SpKL{r}{\alpha}(\hat{\mb P})$ with $|b-a|\geq \Delta'$. Furthermore, as $a,b\in \SpKL{r}{\alpha}(\hat{\mb P})$ there also exists distributions $\mb Q_{a}\in \mc P$ and $\mb Q_{b}\in \mc P$ so that $\KL(\hat {\mb P}, \mb Q_{a})\leq r$ and $\KL(\hat {\mb P}, \mb Q_{b})\leq r$ with $\TV(\mb Q_{a}, \mb P_{a})\leq \alpha$ and $\TV(\mb Q_{b}, \mb P_{b})\leq \alpha$.
  Finally, we consider $\Delta''$ so that we have $\rad{E}{r}{\alpha}<\Delta''<\Delta'<\KLrad{r}{\alpha}$.
  
  We address the cases $r>0$ and $r=0$ separately.\\
  \textbf{Case $r>0$.}  
  Consider the estimator
  \[
    S: \mb P'\mapsto \set{\theta'}{\abs{\theta'-E(\mb P')}\leq \Delta''}
  \]
  which because $\rad{E}{r}{\alpha}<\Delta''$ must verify the guarantee of \cref{eq:point-estimate:guarantee:noise:r}. However, we will show that
  \[
    \max \, \{\liminf_{n \to \infty} \frac{1}{n} \log \Prob_{\mb Q_a}(|E(\mb P_n)- \theta|> \Delta''), \liminf_{n \to \infty} \frac{1}{n} \log \Prob_{\mb Q_{b}}(|E(\mb P_n)+ \theta|> \Delta'')\} \geq -r
  \]
  establishing the contradiction as $\TV(\mb Q_{a}, \mb P_{a})\leq \alpha$ and $\TV(\mb Q_{b}, \mb P_{b})\leq \alpha$. As $\Delta''<\Delta'\leq \abs{b-a}$ it can be remarked that the set $S(\hat {\mb P})$ can not contain the set $\{a, b\}$. Assume here that it does not contain $a$, the other case is completely analagous and is hence omitted. We have here that
  \[
    \hat{\mb P} \in \mc D \defn \set{\mb P'}{|E(\mb P')- a|> \Delta''}.
  \]
  where $\mc D$ is an open set as we assumed that the estimator $E$ is regular. Hence, from Sanov's large deviation lower bound \citep{dembo2009large} it follows that
  \begin{align*}
    \liminf_{n \to \infty} \frac{1}{n} \log \Prob_{\mb Q_a}(|E(\mb P_n)- a|> \Delta'')>-\inf_{\mb P'\in \mc D} \KL(\mb P', \mb Q_{a})\geq - \KL(\hat{\mb P}, \mb Q_a) \geq -r
  \end{align*}
  establishing the claim as $\TV(\mb Q_{a}, \mb P_{a})\leq \alpha$.

  \textbf{Case $r=0$.}
  Observe that we have here $\mb Q_{a}=\mb Q_{b}= \hat{\mb P} $ as $\KL(\hat {\mb P}, \mb Q_{a})=\KL(\hat {\mb P}, \mb Q_{b})=0$. Hence, we have here that $\hat {\mb P}$ satisfies $\TV(\hat{\mb P}, \mb P_{a})\leq \alpha$ and $\TV(\hat{\mb P}, \mb P_{b})\leq \alpha$.
  Consider the estimator
  \[
    S: \mb P'\mapsto \set{\theta'}{\abs{\theta'-E(\mb P')}\leq \Delta''}
  \]
  which because $\rad{E}{0}{\alpha}<\Delta''$ must be feasible in Equation \eqref{eq:point-estimate:guarantee:noise:0}. However, we will show that
  \[
    \max \, \{\liminf_{n \to \infty} \Prob_{\hat{\mb P}}(|E(\mb P_n)- \theta|> \Delta''), \liminf_{n \to \infty} \Prob_{\hat{\mb P}}(|E(\mb P_n)+ \theta|> \Delta'')\} =1
  \]
  establishing the contradiction as $\TV(\hat{\mb P}, \mb P_{a})\leq \alpha$ and $\TV(\hat{\mb P}, \mb P_{b})\leq \alpha$. As $\Delta''<\Delta'$ it can be remarked that the set $S(\hat {\mb P})$ can not contain the set $\{a, b\}$. Assume here that it does not contain $a$, the other case is again completely analagous and is hence omitted. We have here that
  \[
    \hat{\mb P} \in \mc D \defn \set{\mb P'}{|E(\mb P')- a|> \Delta''}.
  \]
  where $\mc D$ is an open set as we assumed that the estimator $E$ is regular.
  Hence, $\tset{\mb P'}{\LP(\mb P, \hat{\mb P})\leq \epsilon}\subseteq \mc D$ for some $\epsilon>0$ sufficiently small.
  Hence, from the weak law of large numbers \citep{varadarajan1958convergence} it follows that
  \begin{align*}
    \liminf_{n \to \infty} \Prob_{\hat {\mb P}}(|E(\mb P_n)- a|> \Delta'')=\liminf_{n \to \infty} \Prob_{\hat {\mb P}}(\mb P_n\in \mc D)>\liminf_{n \to \infty} \Prob_{\hat {\mb P}}(\LP(\mb P_n, \hat{\mb P})\leq \epsilon)=1
  \end{align*}
  establishing the claim as $\TV(\hat {\mb P}, \mb P_{a})\leq \alpha$.
\end{appendixproof}
Now, given the minimality of the DRO confidence set (\cref{thm: KLefficiency:noise}), and its non-convexity, confidence intervals cannot reach minimality in the sense of \cref{thm: KLefficiency:noise}, which is a uniform result across all distributions $\hat{\Pb} \in \cP$. However, we can hope to design confidence intervals that verify a worst-case, rather than uniform, notion of minimality in the sense of reaching the lower bound of \cref{thm:efficiency:location}. We refer to such minimality notion as \textit{worst-case minimal}.

We will show that a worst-case minimal estimator exists and construct it explicitly. 
As we seek a worst-case minimality, an intuitive approach is to construct an estimator tailored to the maximizer $\hat{\Pb}^\star$ in \eqref{eq: KL radius}. 
To investigate such a worst-case distribution, we study the shape of the DRO confidence set for each radius $\Delta$ through the statistical resolution function $\r^{\alpha}(\Delta) := \min\{r \geq 0 \; : \; \exists \hat{\Pb} \in \cP, \;  \{-\Delta,\Delta\} \in \SpKL{\alpha}{r}(\hat{\Pb})\}$. This function characterizes the smallest radius such that the resulting DRO confidence set has radius at least $\Delta$. As a result, it is the smallest statistical resolution for which it is not possible to fence in the unknown parameter within a radius lower than $\Delta$.
Additionally, this statistical resolution function characterizes the worst-case radius given the following result.
\begin{theorem}
  \label{theorem:statistical-resolution-function}
  We have
  \(
    \KLrad{r}{\alpha} = \sup\set{\abs{\Delta}}{\r^{\alpha}(\Delta)\leq r}.
  \)
\end{theorem}
\begin{appendixproof}[Proof of \cref{theorem:statistical-resolution-function}]
  We first show that $\sup\set{\abs{\Delta}}{\r^{\alpha}(\Delta)\leq r} \leq \sup_{\hat {\mb P}\in \mc P} \radius(\SpKL{r}{\alpha}(\hat {\mb P}))$. Take indeed $\Delta'<\sup\set{\abs{\Delta}}{\r^{\alpha}(\Delta)\leq r}$. That is, we have by definition that $\r^{\alpha}(\Delta)\leq r$ for some $\Delta\geq \Delta'$. Hence,
  $\{-\Delta, \Delta\}\subseteq \SpKL{r}{\alpha}(\hat {\mb P})$  and consequently $\Delta'\leq \Delta \leq \radius(\SpKL{r}{\alpha}(\hat {\mb P}))$. As this applies for any $\Delta'<\sup\set{\abs{\Delta}}{\r^{\alpha}(\Delta)\leq r}$ we have that $\sup\set{\abs{\Delta}}{\r^{\alpha}(\Delta)\leq r} \leq \sup_{\hat {\mb P}\in \mc P} \radius(\SpKL{r}{\alpha}(\hat {\mb P}))$.

  We  show that $\sup_{\hat {\mb P}\in \mc P} \radius(S_{r, \alpha}(\hat {\mb P})) \leq \sup\set{\abs{\Delta}}{\r^{\alpha}(\Delta)\leq r}$. Let $\Delta<\sup_{\hat {\mb P}\in \mc P} \radius(\SpKL{r}{\alpha}(\hat {\mb P}))$. There exists $\hat {\mb P}$ so that $\radius(\SpKL{r}{\alpha}(\hat {\mb P}))\geq \Delta$ and 
  hence there exists $a\in \SpKL{r}{\alpha}(\hat {\mb P})$ and $b\in \SpKL{r}{\alpha}(\hat {\mb P})$
  so that $b-a\geq \Delta$. As $\mc P_\Theta$ is a translation invariant family we have that $\hat {\mb P}'$ defined as $\hat {\mb P}'(B) = \hat {\mb P}(B-(a+b)/2)$ for all $B\subseteq\Xi$ satisfies $\{-\Delta, \Delta\} \in \SpKL{r}{\alpha}(\hat {\mb P}') $. 
  Hence, $\r^{\alpha}(\Delta)\leq r$ and consequently $\Delta \leq \sup\set{\abs{\Delta}}{\r^{\alpha}(\Delta)\leq r}$. 
  As this applies for any $\Delta<\sup_{\hat {\mb P}\in \mc P} \radius(\SpKL{r}{\alpha}(\hat {\mb P}))$ we have that $\sup_{\hat {\mb P}\in \mc P} \radius(S_{r, \alpha}(\hat {\mb P})) \leq \sup\set{\abs{\Delta}}{\r^{\alpha}(\Delta)\leq r}$.
\end{appendixproof}

The statistical resolution function can be written as the following distributional optimization problem 
\begin{equation}
  \label{eq:statistical-resolution-function-optimization}
  \begin{array}{rl}
    \r^{\alpha}(\Delta)\defn \min & r\\
    \st & r\geq 0, ~ \hat{\mb P}\in \mc P, ~\mb Q^+_{\Delta} \in \mc P, ~\mb Q^-_{\Delta} \in \mc P, \\
                                 & \KL(\hat{\mb P}, \mb Q^+_{\Delta})\leq r, ~\KL(\hat{\mb P}, \mb Q^-_{\Delta})\leq r,\\
                                 & \TV(\mb Q^+_{\Delta}, \mb P_\Delta)\leq \alpha, ~\TV(\mb Q^-_{\Delta}, \mb P_{-\Delta})\leq \alpha.
  \end{array}
\end{equation}
The statistical resolution function is the smallest Kullback-Leibler distance between two corrupted distributions consistent with two parametric distribution $\mb P_{-\Delta}$ and $\mb P_{\Delta}$ with a common empirical distribution $\hat {\mb P}$; see also Figure \ref{fig:resolution}. The distributions $\Qb^{-}_{\Delta}$ and $\Qb^{+}_{\Delta}$ are the two corrupted distributions from $\mb P_{-\Delta}$ and $\mb P_{\Delta}$ both of which explain the empirical distribution $\hat{\Pb}$ well enough as judged by the Kullback-Leibler divergence.

\begin{figure}[t]
  \centering
  \includegraphics[height=6cm]{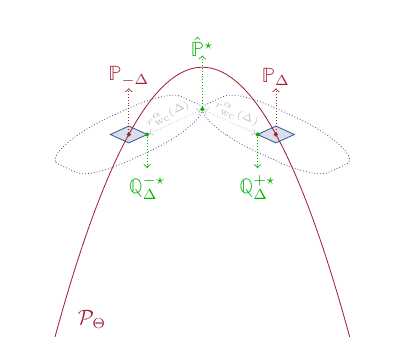}
  \caption{The statistical resolution function $\r^\alpha$ is characterized as the minimum KL distance between a common distribution and the sets $\mc P^-_\Delta\defn \tset{\mb Q_\Delta^-}{\TV(\mb Q_\Delta^-, \mb P_{-\Delta})\leq \alpha}$ and $\mc P^+_\Delta\defn \tset{\mb Q_\Delta^-}{\TV(\mb Q_\Delta^-, \mb P_{-\Delta})\leq \alpha}$ shown in blue. The densely dotted sets represent all distributions at KL distance at most $\r^\alpha(\Delta)$ from the sets $\mc P^-_\Delta$ and $\mc P^+_\Delta$.}
  \label{fig:resolution}
\end{figure}

For all $\Delta>0$ we will consider the minimizers $\hat{\Pb}^\star,\Qb_\Delta^{-\star},\Qb_\Delta^{+\star}$ of problem \eqref{eq:statistical-resolution-function-optimization}. We will show that such ``least favorable'' distributions do exists in Proposition \ref{lemma:huber-optimal-solution} and provide an explicit construction in Equations \eqref{eq:least-favorable-} and \eqref{eq:least-favorable+}. 
Notice that least favorable distributions have a dependence on the corruption parameter $\alpha$ which we do not make explicit. Given a pair of least favorable distributions $\Qb_\Delta^{-\star}$ and $\Qb_\Delta^{+\star}$ we consider an estimator based on a log-likelihood ratio hypothesis test between $\Qb_\Delta^{-\star}$ and $\Qb_\Delta^{+\star}$, i.e., we consider $\varphi_\Delta(\xi) \defn \frac{1}{2}\log(\tfrac{\d\Qb^\star_\Delta}{\d \Qb^\star_{-\Delta}})(\xi)$ for $\xi\in\Re$. The intuition behind this construction is based on the well established connection between hypothesis testing and confidence intervals \citep[Section 10.6]{huber1981robust}. Finally, let us denote with $E^\star_\Delta$ the estimator associated with $\varphi^\star_\Delta$.

\begin{theorem}
  \label{thm:huber-optimal}
  We have that $E^\star_\Delta$ coincides with the Huber estimator $E^{\huber}_{\Delta, k}$ in Example \ref{eg:huber-estimator} for $k = -\log(c')$ where $c'$ is defined implicitely in Equation \eqref{eq:def of c'} for any $\Delta\geq 0$. In particular, for any $\Delta>\KLrad{r}{\alpha}$ and $r>0$ we have that $E^\star_{\Delta}$ satisfies (\ref{eq:point-estimate:guarantee:noise:r}).
\end{theorem}

\begin{appendixproof}[Proof of \cref{thm:huber-optimal}]
  The first part of the theorem is a direct consequence of Proposition \ref{lemma:huber-optimal-solution} which identifies the minimizers $\Qb_\Delta^{-\star}$ and $\Qb_\Delta^{+\star}$ to be of the form given in Equations \eqref{eq:least-favorable-} and \eqref{eq:least-favorable+}, respectively.

  We now show that for any $\Delta>\KLrad{r}{\alpha}$ and $r>0$ we have that $E^\star_{\Delta}$ satisfies (\ref{eq:point-estimate:guarantee:noise:r}).
Let $r\geq 0$ and $\alpha \in [0,1]$.
Remark first that we have
  \[
    \Prob_{\mb P^c}(|E^\star_{\Delta}(\mb P_n)- \theta|> \Delta) = \Prob_{\mb P^c}(E^\star_{\Delta}(\mb P_n)> \theta+ \Delta) + \Prob_{\mb P^c}(\theta -\Delta > E^\star_{\Delta}(\mb P_n))
  \]
  for any distribution $\mb P^c \in \cP$ and $\theta \in \Theta$.
  It can be remarked that for $E^\star_{\Delta}(\mb P_n)> \theta+ \Delta$ it is necessary that $\int \varphi^\star_{\Delta}(\xi-\theta-\Delta)\d\mb P_{n}(\xi) = \sum_{i=1}^n  \varphi^\star_{\Delta}(\xi_i-\theta-\Delta)/n\geq 0$ as $\xi \to \varphi_{\Delta}^\star(\xi)$ is a nondecreasing function as $g$ is logconcave \citep[p. 283]{huber1981robust}.

  Using the translation property of our family we have
  \begin{align*}
    & \forall \theta \in \Theta,~\forall \mb {P}^c \in \mathcal{P},~\TV(\mb {P}^c, \mb P_\theta)\leq \alpha :~\frac 1n \log \Prob_{\mb P^c}(\textstyle\sum_{i=1}^n  \varphi^\star_{\Delta}(\xi_i-\theta-\Delta) /n \geq 0) \leq -r \\
    & \hspace{5em}\iff  \forall \mb Q^-_{\Delta} \in \cP,~\TV(\Qb^-_{\Delta}, \mb P_{-\Delta})\leq \alpha: ~\frac 1n \log \Prob_{\mb Q^-_{\Delta}}(\textstyle\sum_{i=1}^n  \varphi^\star_{\Delta}(\xi_i) /n \geq 0) \leq  -r.
  \end{align*}

  Let the log-moment generating function of the random variable $\varphi_{\Delta}^\star(\xi)$ under a distribution $\mb Q^-_\Delta$ be defined as $\Lambda(t, \mb Q^-_\Delta) \defn \log( \E{\mb Q^-_{\Delta}}{\exp(t \varphi_{\Delta}^\star(\xi))})$ and define the Cram\'er function as its convex conjugate $\Lambda^{\star}(s, \mb Q^-_\Delta) = \sup_{t\in\Re} s t - \Lambda(t, \mb Q^-_\Delta)$. 
  Cram\'er's theorem establishes the upper bound
  \begin{align*}
    \frac 1n \log \Prob_{\mb Q^-_{\Delta}}(\textstyle\sum_{i=1}^n  \varphi^\star_{\Delta}(\xi_i) /n \geq 0) &\leq - \inf_{s\geq 0}\Lambda^{\star}(s,\mb Q^-_\Delta ).
  \end{align*}
  We will now show that $\inf_{s\geq 0}\Lambda^{\star}(s, \mb Q^-_\Delta) \geq \r^{\alpha}(\Delta)$ for any $\mb Q^-_{\Delta} \in \{ \Qb \in \cP \; : \; \TV(\Qb,\Pb_{-\Delta}) \leq \alpha\}$. Indeed, we have
  \begin{align*}
    \inf_{s\geq 0} \Lambda^{\star}(s, \mb Q^-_{\Delta}) = &  \inf_{s\geq 0} \sup_{t\in\Re} s t - \Lambda(t, \mb Q^-_{\Delta}) \geq \inf_{s\geq 0} s - \Lambda(1, \mb Q^-_{\Delta}) \\
    = & - \log\left( \E{\mathbb{Q}^-_{\Delta}}{e^{\varphi^\star_{\Delta}(\xi)}}\right) \\
    = & -\log \int \sqrt{\frac{\d \mb Q_{\Delta}^{+\star}}{\d \mb Q_{\Delta}^{-\star}}(\xi)} \d \mb Q^-_{\Delta}(\xi)  \\
    \geq & -\log \int \sqrt{\frac{\d \mb Q_{\Delta}^{+\star}}{\d \mb Q_{\Delta}^{-\star}}(\xi)} \d \mb Q^{-\star}_{\Delta}(\xi) = -\log\int\sqrt{q^{-\star}_{\Delta}(\xi)q^{+\star}_{\Delta}(\xi)}\, \d\xi = \r^{\alpha}(\Delta)
  \end{align*}
  where the second inequality follows from Lemma \ref{lemma:stochastic-dominance} as $\xi\mapsto \varphi_{\Delta}^\star (\xi)$ is a nondecreasing function and the final equality is shown in \cref{lemma:huber-optimal-solution}.
  We have hence shown that
  \[
    \forall \theta \in \Theta,~\forall \mb {P}^c \in \mathcal{P},~\TV(\mb P^c, \mb P_\theta)\leq \alpha : ~\Prob_{\mb P^c}(E^\star_{\Delta}(\mb P_n)> \theta+ \Delta)\leq \exp(-\r^{\alpha}(\Delta)n).
  \]
  Using a completely symmetrical argument we can similarly establish that also
  \begin{align*}
    \forall \theta \in \Theta,~\forall \mb {P}^c \in \mathcal{P},~\TV(\mb P^c, \mb P_\theta)\leq \alpha : ~\Prob_{\mb P^c}(\theta -\Delta > E^\star_{\Delta}(\mb P_n))\leq \exp(-\r^{\alpha}(\Delta)n).
  \end{align*}
  Putting these inequalities together, and noticing that $r<\r^{\alpha}(\Delta)$, as $\Delta > \KLrad{r}{\alpha}$, it follows that $E^\star_{\KLrad{r}{\alpha}}$ verifies the coverage guarantee \eqref{eq:point-estimate:guarantee:noise}.
\end{appendixproof}

The previous theorem answers a key question of our paper by establishing a natural connection between DRO and robust statistics. In the limit $\Delta\downarrow \KLrad{r}{\alpha}$ the Huber estimator $E^\star_\Delta$ is worst-case minimal achieving the lower bound of \cref{thm:efficiency:location} and hence shares the same minimal radius in the worst-case as the DRO confidence set. However, the Huber estimator lacks uniform minimality, which is in a sense the price to pay for tractability. The following example illustrates this discussion in the case of normal distributions.

\begin{example}[Normal Family Continued]
  \label{ex:normal-family-2}
  Consider the standard normal location family $\mc P_\Theta = \set{N(\theta,1)}{\theta\in \Re}$ parametrized in the unknown mean $\theta\in\Re$ from Example \ref{ex:normal-family}. Let us consider $\alpha=0.1$ and $r=0.43$. We numerically determine in this case $\KLrad{r}{\alpha} = 2$. That is, no errors smaller than $\Delta> \KLrad{r}{\alpha}$ can be resolved in the worst-case with statistical resolution $r=0.43$ in the face of total variation corruption at level $\alpha=0.1$. Figure \ref{fig:ex:least-favorable} depicts the least favorable distributions $\mb Q_{\Delta}^{+\star}$ and $\mb Q_{\Delta}^{-\star}$ as well as the symmetric density $\hat {\mb P}^\star$ which we argued is a maximizer in problem (\ref{eq: KL radius}). That is, we have $\KLrad{r}{\alpha}=\radius(S_{r, \alpha}(\hat {\mb P}^\star))$.

  In Figure \ref{fig:ex:resolution} we depict the statistical resolution function $\theta \mapsto \r^\alpha(\hat {\mb P}^\star, \theta)$ defined in Equation \eqref{eq:resolution_function_set}. Recall that the statistical resolution function characterizes the confidence set as the sublevel set $S_{r, \alpha}(\hat {\mb P}^\star) = \tset{\theta\in \Theta}{\r^\alpha(\hat {\mb P}^\star, \theta)\leq r}$. Hence, in this particular example we numerically determine the minimal confidence set to be the nonconvex set $S_{r, \alpha}(\hat {\mb P}^\star) = [-2, -0.83] \cup [0.83, 2]$.

  By symmetry, we have here that $E^\star_\Delta(\hat {\mb P}^\star) = 0$ for any $\Delta\geq 0$. The Huber confidence set $S^\huber_{\Delta}(\hat {\mb P}^\star) := [E^\star_\Delta(\hat {\mb P}^\star) - \Delta, E^\star_\Delta(\hat {\mb P}^\star) + \Delta]$ indeed achieves the DRO radius $\KLrad{r}{\alpha} = \radius(S_{r, \alpha}(\hat {\mb P}^\star))=2$ when $\Delta \downarrow \KLrad{r}{\alpha}$, by virtue of \cref{thm:huber-optimal}, as visualized in red in Figure \ref{fig:ex:resolution}. On the other hand, in view of \cref{thm: KLefficiency:noise} it should be remarked that even at the worst-case distribution $\hat{\mb P}^\star$, we still have that $[-2, -0.83] \cup [0.83, 2] = S_{r, \alpha}(\hat {\mb P}^\star)\subset [-2, 2] \subseteq S^\huber_{\KLrad{r}{\alpha}}(\hat {\mb P}^\star)$ indicating that although perhaps not as practical the DRO confidence set estimator is smaller than the Huber confidence set estimator.

\end{example}

\begin{figure*}[t!]
    \centering
    \begin{subfigure}[t]{0.5\textwidth}
        \centering
        \includegraphics[height=2.5cm]{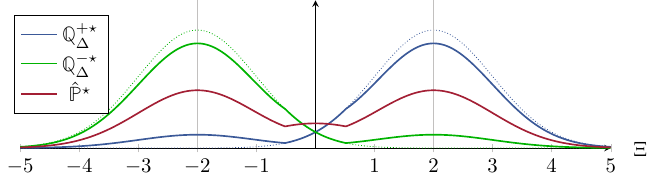}
        \caption{Least favorable distributions}
        \label{fig:ex:least-favorable}
    \end{subfigure}%
    ~ 
    \begin{subfigure}[t]{0.5\textwidth}
        \centering
        \includegraphics[height=5.5cm]{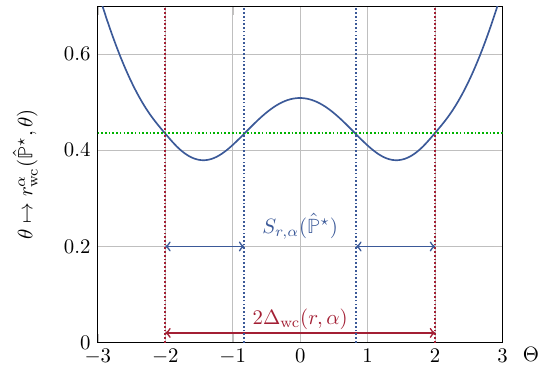}
        \caption{Statistical resolution function}
        \label{fig:ex:resolution}
    \end{subfigure}
    \caption{}
\end{figure*}

We now make explicit the construction of the least-favorable distributions which proves the form of the minimal estimator.

\paragraph{Constructing explicitly the least favorable distributions.}
Let $c_\Delta(\xi) \defn \tfrac{g(\xi-\Delta)}{g(\xi+\Delta)}$ which is a strictly increasing function as $g$ is logconcave \citep[p. 283]{huber1981robust}. 
Consider distributions $\mb Q^{+\star}_{\Delta}$ and $\mb Q_{\Delta}^{-\star}$ defined by their densities
\begin{align}
  q^{-\star}_{\Delta}(\xi) =&
                            \begin{cases}
                              \frac{1}{1+c'} \left[ g(\xi+\Delta) + g(\xi-\Delta) \right] & {\rm{if}}~  \tfrac{g(\xi-\Delta)}{g(\xi+\Delta)}\leq c',\\
                              g(\xi+\Delta)  & {\rm{if}}~c' < \tfrac{g(\xi-\Delta)}{g(\xi+\Delta)} < 1/c' ,\\
                              \frac{c'}{1+c'} \left[ g(\xi+\Delta) + g(\xi-\Delta) \right] & {\rm{if}}~ 1/c'\leq \tfrac{g(\xi-\Delta)}{g(\xi+\Delta)}
                            \end{cases}\label{eq:least-favorable-}\\
  \intertext{and}
  q^{+\star}_{\Delta}(\xi) =&
                           \begin{cases}
                             \frac{c'}{1+c'} \left[ g(\xi+\Delta) + g(\xi-\Delta) \right] & {\rm{if}}~  \tfrac{g(\xi-\Delta)}{g(\xi+\Delta)}\leq c',\\
                             g(\xi-\Delta)  & {\rm{if}}~c' <\tfrac{g(\xi-\Delta)}{g(\xi+\Delta)} < 1/c' ,\\
                             \frac{1}{1+c'} \left[ g(\xi+\Delta) + g(\xi-\Delta) \right] & {\rm{if}}~ 1/c'\leq \tfrac{g(\xi-\Delta)}{g(\xi+\Delta)}
                           \end{cases}\label{eq:least-favorable+}
\end{align}
for the unique constant $c'\leq 1$ chosen to guarantee that $\TV(\mb Q_\Delta^{-\star}, \mb P_{-\Delta})=\alpha$ and $\TV(\mb Q_\Delta^{+\star}, \mb P_{\Delta})=\alpha$ by enforcing the condition
\begin{equation}\label{eq:def of c'}
    f_\Delta(c') \defn \int \left(c' g(\xi+\Delta) - g(\xi-\Delta)\right)^+ \d\xi/(1+c') = \alpha.
\end{equation}

\begin{remark}
Observe that from the symmetry condition on $\mb P_0$ we have that its density function satisfies $g(\xi) = g(-\xi)$ for all $\xi\in \Xi$ from which it follows immediately that
\(
f_\Delta(1) = \int_{\xi< 0} g(\xi+\Delta) \d\xi \geq \frac{1}{2} \geq \alpha.
\)
\citet[p.\ 269]{huber1981robust} shows that the function $f_\Delta(c')$ is continuous and strictly increasing on $c'\geq \inf_{\xi\in \Xi} c_{\Delta}(\xi) = \inf_{\xi\in \Xi} c_{\Delta}(\xi)$ while clearly we have $f_\Delta(c')=0$ for $c'<\inf_{\xi\in \Xi} c_{\Delta}(\xi)$ and as shown before $f(1)\geq \alpha$. Hence, we have that $c'\leq 1$ is uniquely defined for any $\Delta\geq 0$.
\end{remark}

\begin{toappendix}
  \begin{lemma}
    \label{lemma:least-favorable-pair}
    We have that $\mb Q^{-\star}_{\Delta},\mb Q_{\Delta}^{+\star} \in \cP$, $\TV(\mb Q^{-\star}_{\Delta}, \Pb_{-\Delta}) \leq \alpha$ and $\TV(\mb Q^{+\star}_{\Delta}, \Pb_{\Delta}) \leq \alpha$.
  \end{lemma}
  \begin{proof}
  Recall $c_\Delta(\xi) = \tfrac{g(\xi-\Delta)}{g(\xi+\Delta)}$.
    Observe that we have
    \begin{align*}
      q^{-\star}_{\Delta}(\xi) = & \begin{cases}
        \frac{1+ c_\Delta(\xi)}{1+c'} g(\xi+\Delta) & {\rm{if}}~  c_\Delta(\xi)\leq c',\\
        g(\xi+\Delta)  & {\rm{if}}~c' < c_\Delta(\xi) < 1/c' ,\\
        c' \frac{1+c_\Delta(\xi)}{1+c'} g(\xi+\Delta) & {\rm{if}}~ 1/c'\leq c_\Delta(\xi)
      \end{cases}
    \end{align*}
    and hence we have that $q^{-\star}_{\Delta}(\xi) \leq g(\xi+\Delta)$ when $c_\Delta(\xi)\leq c'$, $q^{-\star}_{\Delta}(\xi) = g(\xi+\Delta)$ when $c'< c_\Delta(\xi)< 1/c'$ and finally $q^{-\star}_{\Delta}(\xi) \geq g(\xi+\Delta)$ when $1/c'\leq c_\Delta(\xi)$. Hence, we have that
    \begin{align*}
      & \int \d \mb Q^{-\star}_{\Delta}(\xi)\\
      = & \int q^{-\star}_{\Delta}(\xi) \one{c_\Delta(\xi)\leq c'} \d\xi + \int q^{-\star}_{\Delta}(\xi) \one{c' < c_\Delta(\xi)\leq 1/c'} \d\xi + \int q^{-\star}_{\Delta}(\xi) \one{1/c'\leq c_\Delta(\xi)} \d\xi \\
      =  & \int g(\xi+\Delta) \d\xi + \int  [q^{-\star}_{\Delta}(\xi) - g(\xi+\Delta) ]\one{c_\Delta(\xi)\leq c'} \d\xi \\
      & \quad + \int [q^{-\star}_{\Delta}(\xi) - g(\xi+\Delta)] \one{1/c'\leq c_\Delta(\xi)} \d\xi\\
      = & 1 -\alpha+\alpha=1
    \end{align*}
    where penultimate inequality follows from 
    \begin{align*}
      \alpha = f_\Delta(c') \defn& \int \max(c' g(\xi+\Delta) - g(\xi-\Delta), 0) \d\xi/(1+c') \\
      =& \int [c' g(\xi+\Delta) - g(\xi-\Delta)]\one{c_\Delta(\xi) \leq c'}\d\xi/(1+c') \\
      =& \int [(c'+1) g(\xi+\Delta) - g(\xi+\Delta) - g(\xi-\Delta)] \one{c_\Delta(\xi)\leq c'}\d\xi/(1+c') \\
      =& \int [g(\xi+\Delta) - \frac{1}{1+c'}(g(\xi+\Delta) + g(\xi-\Delta))]\one{c_\Delta(\xi)\leq c'}\d\xi\\
      = & \int [g(\xi+\Delta) - q_{\Delta}^{-\star}(\xi)]\one{c_\Delta(\xi)\leq c'}\d\xi.
    \end{align*}
    and
    \begin{align*}
      \alpha = f_\Delta(c') \defn& \int \max(c' g(\xi+\Delta) - g(\xi-\Delta), 0) \d\xi/(1+c') \\
      = & \int \max(c' g(\xi'-\Delta) - g(\xi'+\Delta), 0) \d\xi'/(1+c') \\
      =& \int [c' g(\xi'-\Delta) - g(\xi'+\Delta) ]\one{1/c' \leq c_\Delta(\xi')}\d\xi'/(1+c') \\
      =& \int [\frac{c'}{1+c'} g(\xi'-\Delta)+\frac{c'}{1+c'} g(\xi'+\Delta) - g(\xi'+\Delta) ]\one{1/c' \leq c_\Delta(\xi')}\d\xi'\\
      = & \int [q_{\Delta}^{-\star}(\xi') - g(\xi'+\Delta)] \one{1/c' \leq c_\Delta(\xi')}\d\xi' = \TV(\mb Q^-_\Delta, \mb P_{-\Delta})
    \end{align*}

    Using these two previous results, we have
    \begin{align*}
        & \TV(\mb Q_\Delta^{-\star}, \mb P_{-\Delta}) \\
        &= \frac{1}{2}\int \left|q_{\Delta}^{-\star} -  g(\xi+\Delta)\right| d\xi \\
        &=
        \frac{1}{2}\int [g(\xi+\Delta) - q_{\Delta}^{-\star}(\xi)]\one{c_\Delta(\xi)\leq c'}\d\xi
        +
        \frac{1}{2}\int [q_{\Delta}^{-\star}(\xi') - g(\xi'+\Delta)] \one{1/c' \leq c_\Delta(\xi')}\d\xi'\\
        &= \alpha
    \end{align*}
    The fact that $\int \d \mb Q^{+\star}_{\Delta}(\xi)=1$ and $\TV(\mb Q_\Delta^{+\star}, \mb P_{\Delta})=\alpha$  can be proven completely analogously.
  \end{proof}

  We will show with the help of a first-order stochastic dominance result by {\citet[p.\ 271]{huber1981robust}} that the pair special least favorable continuous distributions $\mb Q^{-\star}_{\Delta}$ and $\mb Q_{\Delta}^{+\star}$ is optimal in problem (\ref{eq:statistical-resolution-function-optimization}).
  Consider the random variable
\begin{align}
  \label{eq:huber-varphi}
  c^\star_\Delta(\xi) \defn \tfrac{\d \mb Q^\star_\Delta}{\d \mb Q^\star_{-\Delta}}(\xi) = 
  \begin{cases}
    c' & {\rm{if}}~  c_\Delta(\xi)\leq c',\\
    c_\Delta(\xi) & {\rm{if}}~c' < c_\Delta(\xi) < 1/c' ,\\
    1/c' & {\rm{if}}~ c_\Delta(\xi) \geq 1/c'.
  \end{cases}
\end{align}

\begin{lemma} 
  \label{lemma:stochastic-dominance}
  Let $\mc P^-_{\Delta} \defn \tset{\mb Q^-_{\Delta}\in \mc P}{\TV(\mb Q^-_{\Delta}, \mb P_{-\Delta})\leq \alpha}$ and $\mc P^+_{\Delta} \defn \tset{\mb Q^+_{\Delta}\in \mc P}{\TV(\mb Q^+_\Delta, \mb P_{\Delta})\leq \alpha}$.
  The least favorable pair $\mb Q^{-\star}_{\Delta}\in \mc P^-_{\Delta}$ and $\mb Q_{\Delta}^{+\star}\in  \mc P^+_{\Delta}$ satisfies 
  \begin{align*}
    \forall t\geq 0,~\mb Q^-_{\Delta}\in \mc P^-_{\Delta}:& \quad  \textstyle \int \one{c^\star_\Delta(\xi)<t} \d \mb Q^-_{\Delta}(\xi) \geq \int \one{c^\star_\Delta(\xi)<t} \d \mb Q^{-\star}_{\Delta}(\xi),\\
    \forall t\geq 0,~\mb Q^+_{\Delta}\in \mc P^+_{\Delta}: &\quad \textstyle \int \one{c^\star_\Delta(\xi)<t} \d \mb Q^+_{\Delta}(\xi) \leq \int \one{c^\star_\Delta(\xi)<t} \d \mb Q^{+\star}_{\Delta}(\xi).
  \end{align*}
\end{lemma}
\begin{proof}
  Given the definition of the function $c^\star_\Delta(\xi)$ in Equation (\ref{eq:huber-varphi}) the statement is trivially true for any $t\leq c'$ and $t\geq 1/c'$.

  Consider hence an arbitrary $t\in (c', 1/c')$. It can be remarked that
  \begin{align*}
    \int \one{c^\star_\Delta(\xi)<t} \d \mb Q^{-\star}_{\Delta}(\xi) = & \int \one{c^\star_\Delta(\xi)<t} q^{-\star}_{\Delta}(\xi)  \d\xi \\
    = & \int\one{c^\star_\Delta(\xi)\leq c'} q^{-\star}_{\Delta}(\xi) \d \xi+ \int \one{c'<c^\star_\Delta(\xi)<t} q^{-\star}_{\Delta}(\xi) \d \xi\\
    = & \int g(\xi+\Delta) \one{c^\star_\Delta(\xi)\leq c'} \d\xi - \alpha + \int \one{c'<c^\star_\Delta(\xi)<t} g(\xi+\Delta) \d \xi \\
    = & \int g(\xi+\Delta) \one{c^\star_\Delta(\xi)<t} \d\xi -\alpha.
  \end{align*}
  Remark that $\TV(\mb Q^-_{\Delta}, \mb P_{-\Delta})\leq \alpha$ implies that
  \(
  \mb Q_{-\Delta}(E) \leq \mb P_{-\Delta}(E) + \alpha
  \)
  for any set $E\subseteq \Xi$. Hence, with $E= \set{\xi}{c^\star_\Delta(\xi)\geq t}$ we get
  \(
  \int \one{c^\star_\Delta(\xi)\geq t} \d \mb Q_{-\Delta}(\xi) \leq \int \one{c^\star_\Delta(\xi)\geq t} \d \mb P_{-\Delta}(\xi) + \alpha .
  \)
  After taking complements
  \(
  \int \one{c^\star_\Delta(\xi)< t} \d \mb Q_{-\Delta}(\xi) \geq \int \one{c^\star_\Delta(\xi)< t} \d \mb P_{-\Delta}(\xi) - \alpha = \int g(\xi+\Delta) \one{c^\star_\Delta(\xi)<t} \d\xi -\alpha
  \)
  the first claim follows. The second claim can be proven completely analogously.
\end{proof}

The previous two first-order stochastic dominance results are equivalent to the observation that for any increasing measurable function $u:\Re\to\Re$ it follows that we have $\E{\mb Q^-_{\Delta}}{u(c^\star_\Delta(\xi))}\leq \E{\mb Q_{\Delta}^{-\star}}{u(c^\star_\Delta(\xi))}$ and $\E{\mb Q^+_{\Delta}}{u(c^\star_\Delta(\xi))}\geq \E{\mb Q_{\Delta}^{+\star}}{u(c^\star_\Delta(\xi))}$ for all $\mb Q^-_{\Delta}\in \mc P^-_{\Delta}$ and $\mb Q^+_{\Delta}\in \mc P^+_{\Delta}$. This stochastic dominance result allowed \citet{huber1981robust} to derive minimax optimal tests between composite hypothesis classes $\mc P_{-\Delta}$ and $\mc P_{\Delta}$ depicted in Figure \ref{fig:resolution}.
We now show that this result also implies that the pair special least favorable continuous distributions $\mb Q^{-\star}_{\Delta}\in \mc P^-_{\Delta}$ and $\mb Q_{\Delta}^{+\star}\in \mc P^+_{\Delta}$ are optimal in problem (\ref{eq:statistical-resolution-function-optimization}).

\end{toappendix}

\begin{proposition}
  \label{lemma:huber-optimal-solution}
  The minimum in problem (\ref{eq:statistical-resolution-function-optimization}) is achieved at the point
  $(\r^{\alpha}(\Delta), \mb Q^{-\star}_{\Delta}, \mb Q^{+\star}_{\Delta}, \hat{\mb  P}^\star)$ where $\hat{\mb P}^\star$ is a the continuous distribution with symmetric density $$\hat p^\star (\xi) \defn \tfrac{\sqrt{q^{-\star}_{\Delta}(\xi) q^{+\star}_{\Delta}(\xi)}}{\int\! \sqrt{q^{-\star}_{\Delta}(\xi)q^{+\star}_{\Delta}(\xi)}\, \d\xi}$$ for all $\xi\in\Re$ and
  $$\r^{\alpha}(\Delta) = -\log\left(\int\sqrt{q^{-\star}_{\Delta}(\xi)q^{+\star}_{\Delta}(\xi)}\, \d\xi\right).$$
\end{proposition}
\begin{appendixproof}[Proof of \cref{lemma:huber-optimal-solution}]
  Observe that feasibility follows from direct substitution, e.g.,
  \begin{align*}
    \KL(\hat{\mb P}^\star, \mb Q^{-\star}_{\Delta}) = & \int \log\left(\frac{\sqrt{q^{-\star}_{\Delta}(\xi) q^{+\star}_{\Delta}(\xi)}}{q^{-\star}_{\Delta}(\xi)\textstyle\int\! \sqrt{q^{-\star}_{\Delta}(\xi)q^{+\star}_{\Delta}(\xi)}\d\xi} \right) \hat p^\star(\xi) \d \xi \\
    = & -\log\left(\int\sqrt{q^{-\star}_{\Delta}(\xi)q^{+\star}_{\Delta}(\xi)}\, \d\xi\right) +\frac 12 \int \log\left(\frac{q^{+\star}_{\Delta}(\xi)}{q^{-\star}_{\Delta}(\xi)} \right) \hat p^\star(\xi) \d \xi\\
    = & -\log\left(\int\sqrt{q^{-\star}_{\Delta}(\xi)q^{-\star}_{\Delta}(\xi)}\, \d\xi\right) = \r^{\alpha}(\Delta)
  \end{align*}
  where the penultimate equality follows the oddness of the function $c^\star_\Delta(\xi) \defn \tfrac{q^{+\star}_{\Delta}(\xi)}{q^{-\star}_{\Delta}(\xi)}$ and the symmetry of $\hat p^\star$.
  Observe that following the Donsker-Varadhan respresentation of the Kullback-Leibler divergence we have for any $\hat{\mb P}$ with  $\int \varphi^\star_\Delta(\xi)\, \d \hat{\mb P}(\xi) \geq 0$ where $\varphi^\star_\Delta(\xi) \defn \log(c^\star_\Delta(\xi))/2$ that
  \begin{align}
    \KL(\hat{\mb P}, \mb Q^-_{\Delta}) = & \sup_{\Theta\in \mc C_b} \int \Theta(\xi) \d \hat{\mb P}(\xi) - \log \int \exp(\Theta(\xi)) \d \mb Q^-_{\Delta}(\xi)\nonumber\\
    \geq & \int \varphi^\star_\Delta(\xi) \,\d \hat{\mb P}(\xi) - \log \int \exp(\varphi^\star_\Delta(\xi)) \d \mb Q^-_{\Delta}(\xi)\nonumber\\
    \geq & -\log \int \sqrt{c^\star_\Delta(\xi)} \d \mb Q^-_{\Delta}(\xi) \nonumber \\
    \geq & -\log \int \sqrt{c^\star_\Delta(\xi)} \d \mb Q^{-\star}_{\Delta}(\xi) = -\log\int\sqrt{q^{-\star}_{\Delta}(\xi)q^{+\star}_{\Delta}(\xi)}\, \d\xi = \r^{\alpha}(\Delta) \label{eq:kl-inequality-1}
  \end{align}
  where $\mc C_b$ denotes all bounded continuous functions on $\Re$. The first inequality follows from the fact that $\varphi^\star_\Delta\in \mc C_b$. The second inequality follows from the premisse $\int \varphi^\star_\Delta(\xi)\, \d \hat{\mb P}(\xi) \geq 0$. The third inequality follows from Lemma \ref{lemma:stochastic-dominance} as $\xi\mapsto\sqrt{c^\star_\Delta(\xi)}$ is a nondecreasing function.
  Similarly, we have that for any $\hat{\mb P}$ with $ \int \varphi^\star_\Delta(\xi)\, \d \hat{\mb P}(\xi) \leq 0$ and $\mb Q^+_{\Delta}\in \mc P^+_{\Delta}$ that
  \begin{align}
    \KL(\hat{\mb P}, \mb Q^+_{\Delta}) \geq & ~ \r^{\alpha}(\Delta) \label{eq:kl-inequality-2}
  \end{align}
  by a completely analagous symmetric proof.
  
  For the sake of contradiction, assume that we have found $(r',  \hat{\mb P}_\Delta', \mb Q^{+'}_{\Delta}, \mb Q^{-'}_{\Delta})$ with $r'<\r^{\alpha}(\Delta)$ feasible in problem (\ref{eq:statistical-resolution-function-optimization}). That is, we have $\KL(\hat{\mb P}_\Delta', \mb Q^{+'}_{\Delta})\leq r' < \r^{\alpha}(\Delta)$ and hence from inequality \eqref{eq:kl-inequality-1} we must have that $\int \varphi_\Delta^\star(\xi)\, \d \hat{\mb P}'(\xi) < 0$. Similarly,  we have $\KL(\hat{\mb P}_\Delta', \mb Q^{+'}_{\Delta})\leq r' < \r^{\alpha}(\Delta)$ and hence from inequality \eqref{eq:kl-inequality-2} we must also have that $\int \varphi^\star_\Delta(\xi)\, \d \hat{\mb P}'(\xi) > 0$. We have reached a contradiction.
\end{appendixproof}

\subsection{Limiting Regimes}

In this section, we discuss different regimes of estimators depending on the desired statistical resolution $r$ and corruption level $\alpha$. In particular, we show that the mean and median naturally induce worst-case minimal confidence intervals in the cases of $r>0, \alpha=0$ and $r=0,\alpha>0$ respectively, but are generally suboptimal in the case $r,\alpha >0$. Furthermore, we highlight that the statistical resolution of any estimator is limited in the face of corruption at level $\alpha>0$.

\subsubsection{Mean Regime $r>0,\alpha=0$}

We study here the special case $\alpha=0$ in which we assume that the data are directly sampled from the distribution $\mb P_{\theta^\star}$ from which we try to learn $\theta^\star\in \Theta$ without any corruption.
In this regime Equation \eqref{eq:def of c'} implies $c'=0$. Hence, the Huber estimator $E^{\star}_{\Delta}$ for $\Delta\downarrow \KLrad{r}{0}$ which is worst-case minimal by Theorem \ref{thm:huber-optimal} reduces to a generalized estimator estimator.

\begin{definition}[Generalized Mean]
  \label{ex:generalized_mean}
  The generalized mean $E_\Delta^\mean(\mb P_n)$ is the $\varphi$-estimator with $\varphi_\Delta(\xi) \defn \log(\tfrac{g(\xi-\Delta)}{g(\xi+\Delta)})/2$ for all $\Delta\geq 0$.
\end{definition}

Recall here that the log-concave function $g$ denotes the density of the distribution $\mb P_0$ which generates our considered location family as stated in the beginning of Section \ref{sec:location-estimation}.
The classical empirical mean is a $\varphi$-estimator with $\varphi(\xi) = \xi$ for all $\xi \in \Re$ and is the generalized mean in the Gaussian location family discussed in Example \ref{ex: location families} as indeed
$\varphi_\Delta(\xi) = \log(\tfrac{g(\xi-\Delta)}{g(\xi+\Delta)})/2 = 1/4 (\xi-\Delta)^2/\sigma^2-1/4(\xi+\Delta)^2/\sigma^2=\tfrac{\xi\Delta}{\sigma^2}$.

\begin{corollary}
  \label{lemma:wc-density}
  The statistical resolution function satisfies
  \begin{equation}
    \label{eq:Bhattacharyya}
    \r^{0}(\Delta) = -\log\left(\int \sqrt{g(\xi-\Delta)g(\xi+\Delta)}\, \d\xi\right).
  \end{equation}
\end{corollary}
\begin{appendixproof}[Proof of \cref{lemma:wc-density}]
  This is a simple corollary of \cref{lemma:huber-optimal-solution} where here $\mb Q_{\Delta}^{-\star} = \mb P_{-\Delta}^\star$ and $\mb Q^+_{\Delta}=\mb P_\Delta$ associated with the solution $c'=0$ as indeed $f_\Delta(0) = \int\max(- g(\xi-\Delta), 0) \d\xi/(1+c') = 0 = \alpha $.
\end{appendixproof}

\begin{corollary} 
We have that $E^\mean_{\Delta}$ satisfies Equation (\ref{eq:point-estimate:guarantee:noise:r}) for all $\Delta> \KLrad{r}{0}$ with $r>0$.
\end{corollary}

The previous corollary of Theorem \ref{thm:huber-optimal} implies that the generalized mean is indeed minimal as $\Delta\downarrow \KLrad{r}{0}$ in the regime where $r>0$ and $\alpha =0$. In Figure \ref{fig:diagram} we visualize the region in which the generalized mean estimator is worst-case minimal with a green line.

We observe that in the noiseless regime discussed here the statistical resolution function $\r^{0}(\Delta)$ found in Equation (\ref{eq:Bhattacharyya}) is also better known as the Bhattacharyya distance \citep{bhattacharyya1943measure} between the distributions $\mb P_{\Delta}$ and $\mb P_{-\Delta}$ and satisfies $\r^{0}(\Delta)=0$ if and only if $\Delta=0$.
Hence, \cref{theorem:statistical-resolution-function} guarantees that we have the limit $\lim_{r\downarrow 0}\KLrad{r}{0}=\KLrad{0}{0}=0$.
Practically, this means that we can fence in the unknown location parameter within an arbitrarily small confidence set around the generalized mean at the expense of considering a small $r$ and hence weakening our imposed statistical guarantee (\ref{eq:point-estimate:guarantee:noise:r}). 
Likewise, it can also be observed that $\lim_{\Delta\to\infty} \r^{0}(\Delta)=\infty$ and hence $\KLrad{r}{0}<\infty$ for any $r\geq 0$.
Practically, this means that reversely one can attain the statistical guarantee (\ref{eq:point-estimate:guarantee:noise:r}) at any desirable statistical resolution $r$ at the expense of sufficiently inflating the radius of a confidence set around the generalized mean. As we point out in the next two sections neither one of these perhaps intuitive observations remains valid in case of data corruption.

\subsubsection{Median Regime $r=0,\alpha\geq 0$}\label{sec:median}

In the presence of data corruption it will in general not be possible to pin down the unknown location parameter with arbitrary precision as was the case for the generalized mean in the absence of corruption even if the statistical resolution $r$ in the imposed statistical guarantee (\ref{eq:point-estimate:guarantee:noise:r}) tends to zero. Clearly, we have that $\KLrad{r}{\alpha}\geq \KLrad{0}{\alpha}$ for any $r>0$ and hence $\lim_{r\downarrow 0}\KLrad{r}{\alpha}\geq \KLrad{0}{\alpha}$. 

Remark that in \cref{thm:huber-optimal} we have excluded the corner case $r=0$ which we will now study in more detail.
The following result taken together with the lower bound of \cref{thm:efficiency:location} guarantees that the median estimator, defined in \cref{exp:median}, is worst-case minimal in this particular setting.

\begin{proposition}
  \label{prop:lower-bound-gross-error-margin-noise}
  We have $\rad{E^\median}{0}{\alpha} = \KLrad{0}{\alpha}$ for all $\alpha\geq 0$. 
\end{proposition}
\begin{appendixproof}[Proof of \cref{prop:lower-bound-gross-error-margin-noise}]
 We shall prove that the median estimator satisfies \eqref{eq:point-estimate:guarantee:noise:0} for all $\Delta> \KLrad{0}{\alpha}$. Hence, taking $\Delta\downarrow \KLrad{0}{\alpha}$ establishes the claim.

  Fix $\Delta> \KLrad{0}{\alpha}$.
  As $\Delta> \KLrad{0}{\alpha}$ we first establish that we must have $\int \one{\xi<0} \d \mb P_{-\Delta}(\xi)-\alpha>\tfrac 12$ and by symmetry $\int \one{\xi<0} \d \mb P_{\Delta}(\xi)+\alpha <\tfrac 12$. Indeed, suppose we have for the sake of contradiction $\int \one{\xi<0} \d \mb P_{-\Delta}(\xi)-\alpha\leq \tfrac 12$ and consequently by symmetry also $\int \one{\xi<0} \d \mb P_{\Delta}(\xi)+\alpha \geq \tfrac 12$. Consider the distributions
  \(
    \mb Q^{-}_{\Delta} = \mb Q^{+}_{\Delta} = (\mb P_{-\Delta}+\mb P_{\Delta})/2.
  \)
  We have
  \begin{align*}
    & \TV(\mb Q^{-}_{\Delta}, \mb P_{-\Delta})\\
    = & \frac{1}{2} \int \abs{g(\xi+\Delta)- g(\xi+\Delta)/2- g(\xi-\Delta)/2 } \d \xi \\
    = &\frac{1}{4} \int \abs{g(\xi+\Delta)- g(\xi-\Delta) } \d \xi \\
    = & \frac{1}{4} \int \one{\xi<0} (g(\xi+\Delta)- g(\xi-\Delta))  \d \xi + \frac{1}{4} \int \one{\xi>0} (g(\xi-\Delta)- g(\xi+\Delta))  \d \xi\\
    = & \frac{1}{2} \left(\int \one{\xi<0} \d \mb P_{-\Delta}(\xi) - \int \one{\xi<0} \d \mb P_{\Delta}(\xi)\right)\leq \frac 12 (1/2+\alpha -1/2 +\alpha)=\alpha
  \end{align*}
  where the fourth equality uses the symmetry of $g$.
  We can similarly also establish that $\TV(\mb Q^{+}_{\Delta}, \mb P_{\Delta}) \leq \alpha$. However, this shows that $\r^{\alpha}(\Delta)=0$ and hence $\Delta\leq \KLrad{0}{\alpha}$; a contradiction.

  With this result out of the way, consider the least favorable pair $\mb Q_\Delta^{-\star}$ and $\mb Q_\Delta^{-\star}$ defined in Section \ref{sec:fund-lower-bound}, and the increasing function $\xi\mapsto c_\Delta(\xi) \defn \tfrac{g(\xi-\Delta)}{g(\xi+\Delta)}$.
  Recall from the proof of Lemma \ref{lemma:least-favorable-pair} that
  the condition $f_\Delta(c') \defn \int \max(c' g(\xi+\Delta) - g(\xi-\Delta, 0) \d\xi/(1+c') = \alpha$ ensures that
  \begin{align*}
    & \int \abs{ g(\xi+\Delta) - q^{-\star}_{\Delta}(\xi) } \one{c_\Delta(\xi)\leq c'} \d \xi = \alpha,
     \int \abs{ g(\xi-\Delta) - q^{+\star}_{\Delta}(\xi) } \one{c_\Delta(\xi)\leq c'} \d \xi = \alpha.
  \end{align*}
  Apply now Lemma \ref{lemma:stochastic-dominance} at $t=1$. We obtain the guarantee
  \begin{align*}
   \forall \mb Q^-_{\Delta}\in \mc P^-_{\Delta}:&~  \textstyle \int \one{\xi<0} \d \mb Q^-_{\Delta}(\xi) = \int \one{c_\Delta(\xi)<1} \d \mb Q^-_{\Delta}(\xi)  \geq \int \one{c_\Delta(\xi)<1} \d \mb Q^{-\star}_{\Delta}(\xi) \\
    & \quad = \int \one{\xi<0} \d \mb Q^{-\star}_{\Delta}(\xi) = \int \one{\xi<0} \d \mb P_{-\Delta}(\xi)-\alpha>\frac 12.
  \end{align*}
  
  Remark now that we can write
  \[
    \Prob_{\mb P^c}(|E^\median(\mb P_n)- \theta|> \Delta) = \Prob_{\mb P^c}(E^\median(\mb P_n)> \theta+ \Delta) + \Prob_{\mb P^c}(\theta -\Delta > E^\median(\mb P_n))
  \]
  for any $\mb P^c$.
  It can be remarked that for $E^\median(\mb P_n)> \theta+ \Delta$ it is necessary that  we have $\int \one{\xi-\theta-\Delta\geq 0}\d\mb P_{n}(\xi) \geq 1/2$. Hence, we have 
  $$\Prob_{\mb P^c}(E^\median(\mb P_n)> \theta+ \Delta) \leq \Prob_{\mb P^c}(\textstyle\sum_{i=1}^n  \one{\xi_i-\theta-\Delta\geq  0} /n \geq \frac 12).$$
  Using the translation property of our location family we have
  \begin{align*}
    & \forall \theta \in \Theta,~\forall \mb {P}^c \in \mathcal{P},~\TV(\mb Q, \mb P_\theta)\leq \alpha :~\limsup_{n\to\infty} \Prob_{\mb P^c}(\textstyle\one{\xi_i-\theta-\Delta\geq 0} /n \geq \frac 12)=0 \\
    & \hspace{11em}\iff  \forall \mb Q^-_{\Delta} \in \mathcal{P}^-_{\Delta}: ~\limsup_{n\to\infty} \Prob_{\mb Q^-_{\Delta}}(\textstyle\sum_{i=1}^n  \one{\xi_i\geq 0} /n \geq \frac 12) =0.
  \end{align*}

  Remark that the random variables $\one{\xi_i\geq 0}$ are independent Bernouilli random variables. 
  Applying Chernoff's for Bernouilli random variables yields the tail bound
  \begin{align*}
    & \frac 1n \log \Prob_{\mb Q^-_{\Delta}}(\textstyle\sum_{i=1}^n  \one{\xi_i\geq 0} /n \geq \frac 12)
    \leq  -\frac{1}{2} \log\left(\frac{\tfrac 12}{\int \one{\xi<0} \d \mb Q^{-}_{\Delta}(\xi) }\right)-\frac 12 \log\left(\frac{\tfrac 12}{1-\int \one{\xi<0} \d \mb Q^{-}_{\Delta}(\xi)}\right)<0
  \end{align*}
  as indeed we have established earlier $\int \one{\xi<0} \d \mb Q^{-}_{\Delta}(\xi) > \frac 12$ for all $\mb Q^{-}_{\Delta}\in \mc P_{-\Delta}$.
  We have hence shown that for all $\theta \in \Theta$, $\mb {P}^c \in \mathcal{P}$ with $\TV(\mb {P}^c, \mb P_\theta)\leq \alpha$ we have that $\frac 1n \Prob_{\mb {P}^c}(E^\median(\mb P_n)> \theta+ \Delta)=0$ for all $n\geq 1$. Using an analagous argument we can establish that for all $\theta \in \Theta$, $\mb {P}^c \in \mathcal{P}$ with $\TV(\mb {P}^c, \mb P_\theta)\leq \alpha$ we have that also $\frac 1n \log \Prob_{\mb {P}^c}(\theta -\Delta > E^\median(\mb P_n))<0$ for all $n\geq 1$. Combining both results we get
  \[
    \forall \theta \in \Theta,~\forall \mb {Q} \in \mathcal{P},~\TV(\mb {P}^c, \mb P_\theta)\leq \alpha : ~\frac 1n  \log \Prob_{\mb {P}^c}(|E^\median(\mb P_n)- \theta|> \Delta) <0 \quad \forall n\geq 1
  \]
  and consequently the median estimator satisfies \eqref{eq:point-estimate:guarantee:noise:0}.
\end{appendixproof}

In Figure \ref{fig:diagram} we visualize the region in which the median estimator is worst-case minimal with a blue line. Consider here the distributions $\hat{\mb P}=\mb Q^+_\Delta=\mb Q^-_\Delta = \mb P_0$ where $\mb P_0$ is the distributions generating our location family. Trivially, we have $\r^{\alpha}(\Delta)=0$ for any $\Delta\geq 0$ for which $\TV(\mb P_0, \mb P_{-\Delta}) = \TV(\mb P_0, \mb P_{\Delta}) \leq \alpha$. Consequently, from \cref{theorem:statistical-resolution-function} it follows hence that for $\alpha>0$ we have
\[
  \lim_{r\downarrow 0}\KLrad{r}{\alpha} \geq \KLrad{0}{\alpha} \geq \sup\set{\Delta\geq 0}{\TV(\mb P_0, \mb P_{\Delta})\leq \alpha}>0.
\]
Practically, this means due to data corruption we can not recover the unknown location parameter precisely as was the case in the noiseless regime.

\subsubsection{Unreachable Regime $r\to \infty$, $\alpha>0$}
\label{sec:extreme-regime-case}

\begin{figure}[t]
  \centering
  \includegraphics[width=0.6\textwidth]{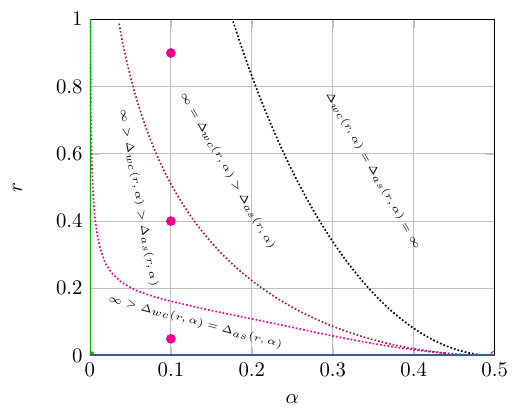}
  \caption{Phase diagram in the statistical resolution / corruption level $(r, \alpha)$ space for a standard normal location family.
  The green (respectively blue) region indicates the regime when the generalized mean (respectively median) is worst-case minimal. The dotted lines indicate phase transitions in terms of worst-case radius $\KLrad{r}{\alpha}$ and almost sure minimal radius $\KLradp{r}{\alpha}$. 
  In particular, the red and black dotted lines indicate the $(r,\alpha)$ phases with existing confidence intervals with non-trivial worst-case radius and almost sure radius respectively.}
  \label{fig:diagram}
\end{figure}

Recall that in the absence of data corruption ($\alpha=0$) one can attain the statistical guarantee (\ref{eq:point-estimate:guarantee:noise:r}) at any desirable statistical resolution $r\geq 0$ by sufficiently inflating the radius of the confidence set estimates based on for instance the generalized mean. We shall show here that when data corruption is present we may have that $\KLrad{r}{\alpha} = \infty$ for $r\in\Re$ which implies that the statistical guarantee (\ref{eq:point-estimate:guarantee:noise:r}) can not be attained by any regular confidence set predictor with bounded confidence sets estimates.

To do so, for any $\Delta$, we will construct a feasible solution in Equation \eqref{eq:statistical-resolution-function-optimization} with a finite objective less than $\bar{r}(\alpha)\defn \frac 12 \log\left(\tfrac{1}{(2\alpha)}\right) + \frac{1}{2} \log\left(\tfrac{1}{(2(1-\alpha))}\right)$. This implies $\r^\alpha(\Delta) < \bar{r}(\alpha)$ for any $\Delta$, and therefore $\KLrad{r}{\alpha} = \infty$ for any $r \geq \bar{r}(\alpha)$ using \cref{theorem:statistical-resolution-function}.
Let us introduce the distributions
\begin{align*}
  \mb Q^-_{\Delta} = & (1-\alpha) \mb P_{-\Delta} + \alpha \mb P_{\Delta},\\
  \mb Q^+_{\Delta} = & (1-\alpha) \mb P_{\Delta} + \alpha \mb P_{-\Delta},\\
  \hat{\mb P} = & (\mb P_{\Delta} + \mb P_{-\Delta})/2.
\end{align*}
It can easily be verified that indeed $\TV(\mb P_{-\Delta}, \mb Q^-_{\Delta} ) = \TV(\mb P_{-\Delta}, (1-\alpha)\mb P_{-\Delta}+\alpha \mb P_{\Delta}) \leq (1-\alpha)\TV(\mb P_{-\Delta}, \mb P_{-\Delta})  +\alpha \TV(\mb P_{-\Delta}, \mb P_{\Delta}) \leq \alpha$ and symmetrically $\TV(\mb P_{\Delta}, \mb Q^+_{\Delta} )\leq \alpha$. Observe now that
\begin{align*}
  \lim_{\Delta\to\infty} \KL((\mb P_{\Delta} + \mb P_{-\Delta})/2, \mb Q^-_{\Delta}) = & \lim_{\Delta\to\infty} \KL((\mb P_{\Delta} + \mb P_{-\Delta})/2, (1-\alpha) \mb P_{-\Delta} + \alpha \mb P_{\Delta})\\
  = & \frac 12 \log\left(\frac{1}{2\alpha}\right) + \frac{1}{2} \log\left(\frac{1}{2(1-\alpha)}\right) = \bar r(\alpha).
\end{align*}

Hence, \cref{theorem:statistical-resolution-function} implies the following result.

\begin{proposition}
    We have $\KLrad{r}{\alpha} = \infty$ for any $r\geq\bar r(\alpha) = \frac 12 \log\left(\tfrac{1}{(2\alpha)}\right) + \frac{1}{2} \log\left(\tfrac{1}{(2(1-\alpha))}\right)$.
\end{proposition}

Remark in particular that the previous observation implies $\KLrad{0}{\tfrac 12} = \infty$. 
This is not surprising, if more than half of the distribution mass is corrupted, nothing much can be said from data.
We have illustrated this observation with the help of a red dotted line feasibility line in Figure \ref{fig:diagram}.

\section{Almost Sure Minimality and Phase Transitions}
\label{sec:eventual-efficiency}

We would like here to reinterpret the efficiency notions we have encountered in Section \ref{sec:parametric-estimation} and Section \ref{sec:location-estimation} in terms of the asymptotic behavior of our proposed confidence set predictors. Let us consider hence the event in which $\mb P_n\to\mb P_\infty$.
Remark first that for any regular set estimator $S$ it follows from its continuity condition that we have
\(
\lim_{n\to\infty} S(\mb P_n) = S(\mb P_\infty).
\)
That is, in a large sample limit the predicted confidence sets $S(\mb P_n)$ for $n\geq 1$ tend to the confidence set $S(\mb P_\infty)$ associated with the distribution $\mb P_\infty$. 
Recall that although the DRO minimal confidence set estimator $\SpKL{r}{\alpha}$ itself does not satisfy the imposed statistical guarantee (\ref{eq:feasibility:parametric}), we proposed a particular $0<\delta$-inflation $\SpKL{r}{\alpha}^\delta$ confidence set estimator which does. We are interested in the large sample behavior of these confidence set estimators for $\delta\downarrow 0$ which is characterized by the following result.

\begin{lemma}
  \label{lemma:approximation:A} 
  The event $\mb P_n\to\mb P_\infty$ implies
  $\lim_{\delta\downarrow 0} \lim_{n\to\infty} \SpKL{r}{\alpha}^\delta(\mb P_n)= \SpKL{r}{\alpha}(\mb P_\infty)$.
\end{lemma}
\begin{appendixproof}[Proof of Theorem \ref{lemma:approximation:A}]
  By \cite{varadarajan1958convergence} we have that $\lim_{n\to\infty}\LP(\mb P_n, \mb P^c)=0$  almost surely. There exists a random time $n_1$ so that $\LP(\mb P_n, \mb P^c)\leq \delta$ for $n\geq n_1$ almost surely.
  We have from the triangle inequality $\LP(\mb P'', \mb P_n)+\LP(\mb P_n, \mb P^c)\geq \LP(\mb P'', \mb P^c)$ that
  \begin{align*}
    & \SpKL{r}{\alpha}^\delta(\mathbb{P}_n)\\
    = &\set{\theta\in \Theta}{\exists \mb Q\in\mc P, ~\mb P''\in \mc P ~{\rm{st}}~\LP(\mb P_n, \mb P'')\leq \delta, ~\KL(\mb P'', \mb Q)\leq r, ~\TV(\mb Q, \mb P_\theta)\leq \alpha}\\
    \subseteq & \set{\theta\in \Theta}{\exists \mb Q\in\mc P,~\mb P''\in \mc P ~{\rm{st}}~\LP(\mb P^c, \mb P'')\leq \delta + \LP(\mb P_n, \mb P^c), ~\KL(\mb P'', \mb Q)\leq r, ~\TV(\mb Q, \mb P_\theta)\leq \alpha}\\
    = & \SpKL{r}{\alpha}^{\delta+\LP(\mb P_n, \mb P^c)}(\mb P^c) \subseteq \SpKL{r}{\alpha}^{2\delta}(\mb P^c)
  \end{align*}
  for $n\geq n_1$. Furthermore, we also have that when $\LP(\mb P_n, \mb P^c)\leq \delta$ clearly we get  $\SpKL{r}{\alpha}(\mathbb{P}^c) \subseteq \SpKL{r}{\alpha}^\delta(\mathbb{P}_n)$. Hence, we have almost surely the sandwich inequality
  \[
    \SpKL{r}{\alpha}(\mathbb{P}^c) \subseteq \lim_{n\to\infty} \SpKL{r}{\alpha}^\delta(\mb P_n)\leq \SpKL{r}{\alpha}^{2\delta}(\mb P^c).
  \]
  The statement now follows from \cref{thm:approximation:B} by taking the limit $\delta\downarrow 0$.
\end{appendixproof}

As an immediate consequence of the uniform minimality result in \cref{thm: KLefficiency:noise} we have that for any regular estimator $S$ which satisfies our imposed guarantee \eqref{eq:feasibility:parametric} that
\begin{equation}
  \label{eq:eventual-efficiency}
  \lim_{\delta\downarrow 0} \lim_{n\to\infty} \SpKL{r}{\alpha}^\delta(\mb P_n) = \SpKL{r}{\alpha}(\mb P_\infty)  \subseteq \lim_{n\to\infty} S(\mb P_n) = S(\mb P_\infty) \quad \forall \mb P_\infty\in \mc P.
\end{equation}
Informally, by choosing the inflation $\delta>0$ sufficiently small, the feasible confidence estimator $\SpKL{r}{\alpha}^\delta$ eventually fences in, for large sample sizes, the unknown distribution within a smaller confidence region than any other regular estimator. Inequality \eqref{eq:eventual-efficiency} simply reinterprets why the proposed minimal estimator is to be preferred from the perspective of the asymptotic region within which its fences in the unknown location parameter.

We have shown earlier that the Huber confidence set estimator $S^{\huber}_\Delta$ of radius $\Delta>\KLrad{r}{\alpha}$ around the Huber estimator also does enjoy the statistical guarantee (\ref{eq:feasibility:parametric}).
At the same time, for any $\Delta'<\KLrad{r}{\alpha}$ we discussed an explicit construction of a distribution $\hat{\mb{P}}^\star$ so that $\Delta'\leq \radius(\SpKL{r}{\alpha}(\hat{\mb{P}}^\star))$. In a worst-case scenario when $\mb P_n\to \hat{\mb{P}}^\star$ occurs, we get
\[
\radius(S^{\huber}_\Delta({\mb P}_n )) = \Delta > \KLrad{r}{\alpha} \geq \radius(S_{r, \alpha}(\hat{\mb P}^\star)) = \radius(\textstyle\lim_{\delta\downarrow 0} \lim_{n\to\infty}S^\delta_{r, \alpha}(\mb P^n)) \geq \Delta'.
\]
Hence, as we may let $\Delta\downarrow \KLrad{r}{\alpha}$ and $\Delta'\uparrow\KLrad{r}{\alpha}$, the worst-case efficiency of the Huber estimator can be reinterpreted from a perspective of enjoying a worst-case asymptotic region (in terms of radius) within which its fences in the unknown location parameter.

A decision-maker may however reasonably decide that worrying about a worst-case event occurring is excessive.
Define indeed here $\mc P_\Theta^\alpha\defn \set{\mb P^c\in \mc P}{\exists \theta\in\Theta:~\TV(\mb P^c, \mb P_\theta)\leq \alpha}$ as the set of all potential corrupted distributions.
Let us consider a case in which $\hat{\mb{P}}^\star \not\in \mc P_\Theta^\alpha$ and hence the law of large numbers \citep{varadarajan1958convergence} guarantees that almost surely the worst-case scenario $\mb P_n\to \hat{\mb{P}}^\star$ does in fact not occur.
The previous remark puts into question the significance of the worst-case minimality of the Huber confidence estimator as the worst-case scenarios to which it lends its minimality may not occur almost surely. Hence, when interested in almost-sure rather than worst-case performance one may simply exclude the case $\hat{\mb{P}}^\star \not\in \mc P_\Theta^\alpha$ and observe indeed that
\begin{equation}\label{eq: KL radius prime}
  \lim_{\delta\downarrow 0} \lim_{n\to\infty} \radius(\SpKL{r}{\alpha}^\delta(\mb P_n))= \SpKL{r}{\alpha}(\mb P^c) \leq \KLradp{r}{\alpha} \defn \textstyle\sup_{\hat{\mb P} \in \mc P_\Theta^\alpha}\radius(\SpKL{r}{\alpha}(\hat{\mb P}))
\end{equation}
as almost surely $\mb P_n\to \mb P^c$ where we recall that $\mb P^c$ denotes the corrupted distribution from which the training data points are sampled.
The radius $\KLradp{r}{\alpha}$ can be thought here of as the almost sure counterpart to the worst-case radius $\KLrad{r}{\alpha}$ as it upper bounds the almost sure rather than the worst-case radius of the confidence sets returned by our minimal confidence set estimator in the large sample limit.
It is trivial to observe that by definition we have the inequality $\KLradp{r}{\alpha} \leq \KLrad{r}{\alpha}$. By the end of this section we will distinguish several regimes based on whether the almost-sure and worst-case radii take on finite values and in particular whether they are equal.

However, before we do so it is once more of interest to understand to first introduce the statistical resolution function
\begin{equation}
  \label{eq:statistical-resolution-function-optimization-3}
  \begin{array}{rl}
    \rp^{\alpha}(\Delta)\defn \min & r\\
    \st & r\geq 0, ~\theta\in[-\Delta, \Delta],~\hat{\mb P}\in \mc P, ~\mb Q^+_{\Delta} \in \mc P, ~\mb Q^-_{\Delta} \in \mc P,\\
                                         & \KL(\hat{\mb P}, \mb Q^+_{\Delta})\leq r, ~\KL(\hat{\mb P}, \mb Q^-_{\Delta})\leq r,\\
                                         & \TV(\mb Q^+_{\Delta}, \mb P_\Delta)\leq \alpha, ~\TV(\hat{\mb P}, \mb P_{\theta})\leq \alpha,~\TV(\mb Q^-_{\Delta}, \mb P_{-\Delta})\leq \alpha.
  \end{array}
\end{equation}

The next theorem establishes that the introduced statistical resolution function characterizes the quantity $\KLradp{r}{\alpha}$ as its sublevel sets and should be thought of as the almost sure counterpart to the statistical resolution function introduced in Equation (\ref{eq:statistical-resolution-function-optimization}).

\begin{theorem}
  \label{theorem:statistical-resolution-function-3}
  We have
  \(
  \KLradp{r}{\alpha} = \sup\set{\abs{\Delta}}{\rp^{\alpha}(\Delta)\leq r}.
  \)
\end{theorem}

\begin{appendixproof}[Proof of \cref{theorem:statistical-resolution-function-3}]

  We first show that $\sup\set{\abs{\Delta}}{\rp^{\alpha}(\Delta)\leq r} \leq \sup \tset{ \radius(\SpKL{r}{\alpha}(\hat {\mb P}))}{\hat {\mb P}\in \mc P_\Theta^\alpha}$. Take indeed $\Delta'<\sup\set{\abs{\Delta}}{\rp^{\alpha}(\Delta)\leq r}$. That is, we have that $\rp^{\alpha}(\Delta)\leq r$ for some $\Delta\geq \Delta'$. Hence, by definition of $\r^{\alpha}(\Delta)$
  there exists a $\hat {\mb P}\in \mc P$ and $\theta\in [-\Delta, \Delta]$ so that $\KL(\hat {\mb P}, \mb Q^-_{\Delta})\leq r$ and $\KL(\hat {\mb P}, \mb Q^+_{\Delta})\leq r$ with $\TV(\mb Q^-_{\Delta}, \mb P_{-\Delta})\leq \alpha$, $\TV(\mb Q^+_{\Delta}, \mb P_{\Delta})\leq \alpha$ and $\TV(\hat{\mb P}, \mb P_{\theta})\leq \alpha$. Hence, it follows that $\{-\Delta ,\theta, \Delta \}\subseteq  \SpKL{r}{\alpha}(\hat {\mb P})$  and consequently $\Delta'\leq \Delta \leq \radius(\SpKL{r}{\alpha}(\hat {\mb P}))$. As this applies for any $\Delta'<\sup\set{\abs{\Delta}}{\rp^{\alpha }(\Delta)\leq r}$ we have that $\sup\set{\abs{\Delta}}{\rp^{\alpha}(\Delta)\leq r} \leq \sup \tset{ \radius(\SpKL{r}{\alpha}(\hat {\mb P}))}{\theta\in \Theta, ~\hat {\mb P}\in \mc P, ~\TV(\hat {\mb P}, \mb P_\theta)\leq \alpha}$.

We now show that $\sup \tset{ \radius(\SpKL{r}{\alpha}(\hat {\mb P}))}{\hat {\mb P}\in \mc P_\Theta^\alpha} \leq \sup\set{\abs{\Delta}}{\rp^{\alpha}(\Delta)\leq r}$. Consider here $\Delta<\sup \tset{ \radius(\SpKL{r}{\alpha}(\hat {\mb P}))}{\hat {\mb P}\in \mc P_\Theta^\alpha}$.
There exists now $\hat {\mb P}\in \mc P_\Theta^\alpha$ so that so that $\radius(\SpKL{r}{\alpha}(\hat {\mb P}))\geq \Delta$.
By definition of the set $\mc P_\Theta^\alpha$ there also exists $\theta\in\Theta$ so that $\TV(\hat {\mb P}, \mb P_\theta)\leq \alpha$ and consequently $\theta\in \SpKL{0}{\alpha}(\hat {\mb P})\subseteq \SpKL{r}{\alpha}(\hat {\mb P})$.
We can hence find two points $a\in \SpKL{r}{\alpha}(\hat {\mb P})$ and $b\in \SpKL{r}{\alpha}(\hat {\mb P})$ so that $a\leq \theta$, $b\geq \theta$ and $b-a\geq \Delta$.
As $\mc P_\Theta$ is a translation invariant family we have that $\hat {\mb P}'\in \mc P$ defined as $\hat {\mb P}'(B) = \hat {\mb P}(B-(a+b)/2)$ for all $B\subseteq\Xi$ satisfies $\{-\Delta, \Delta\} \in \SpKL{r}{\alpha}(\hat {\mb P}') $ and $\TV(\hat {\mb P}, \mb P_{\theta'})\leq \alpha$ for $\theta'\in [-\Delta, \Delta]$. Consequently, by definition of the KL-TV confidence set we can find $\mb Q^-_{\Delta}$ with $\TV(\mb Q^-_{\Delta}, \mb P_{-\Delta})\leq \alpha$ and $\mb Q^+_{\Delta}$ with $\TV(\mb Q^+_{\Delta}, \mb P_{\Delta})\leq \alpha$ so that $\KL(\hat{\mb P}', \mb Q^-_{\Delta})\leq r$ and $\KL(\hat{\mb P}', \mb Q^+_{\Delta})\leq r$ with $\TV(\hat {\mb P}, \mb P_{\theta'})\leq \alpha$ for $\theta'\in [-\Delta, \Delta]$.
   Hence, $\rp^{\alpha}(\Delta)\leq r$ and consequently $\sup \tset{ \radius(\SpKL{r}{\alpha}(\hat {\mb P}))}{\hat {\mb P}\in \mc P_\Theta^\alpha} \leq \sup\set{\abs{\Delta}}{\rp^{\alpha}(\Delta)\leq r}$.
\end{appendixproof}

In what follows we compare the worst-case efficiency notion discussed in Section \ref{sec:location-estimation} with the almost sure minimality notion discussed here and identify three different regimes; see also Figure \ref{fig:diagram}.

\begin{sidewaysfigure}
  \centering
  \begin{subfigure}[b]{0.65\textwidth}
    \centering
      \includegraphics[height=2.5cm]{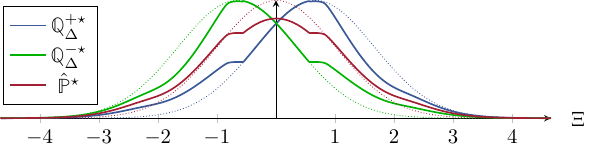}
    \caption{Worst-case distributions $\KLradp{r}{\alpha}=\KLrad{r}{\alpha}$}
    \label{fig:ex:eventual-distributions-i}
  \end{subfigure}%
  \hfill
  \begin{subfigure}[b]{0.35\textwidth}
    \centering
      \includegraphics[height=5cm]{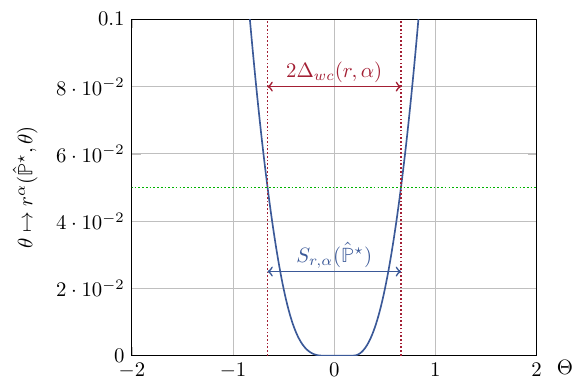}
    \caption{Resolution function $\KLradp{r}{\alpha}=\KLrad{r}{\alpha}$}
    \label{fig:ex:eventual-residual-i}
  \end{subfigure}%
  \\
  \begin{subfigure}[b]{0.65\textwidth}
    \centering
      \includegraphics[height=2.5cm]{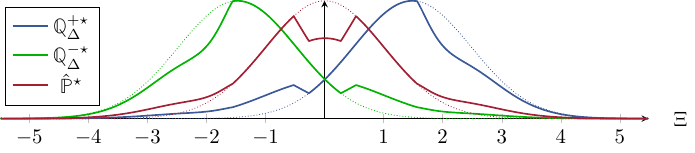}
    \caption{Worst-case distributions $\KLradp{r}{\alpha}<\KLrad{r}{\alpha}<\infty$}
    \label{fig:ex:eventual-distributions-ii}
  \end{subfigure}%
  \hfill
  \begin{subfigure}[b]{0.35\textwidth}
    \centering
      \includegraphics[height=5cm]{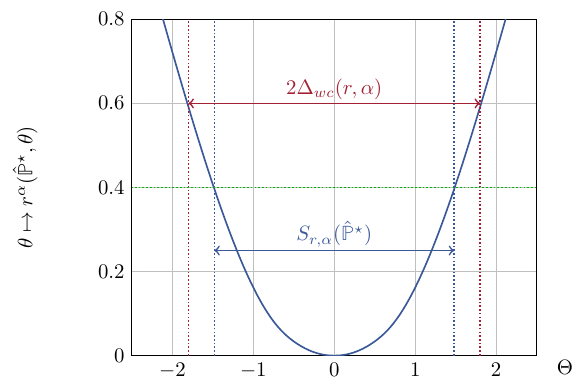}
    \caption{Resolution function $\KLradp{r}{\alpha}<\KLrad{r}{\alpha}<\infty$}
    \label{fig:ex:eventual-residual-ii}
  \end{subfigure}%
  \\
  \begin{subfigure}[b]{0.65\textwidth}
    \centering
      \includegraphics[height=2.5cm]{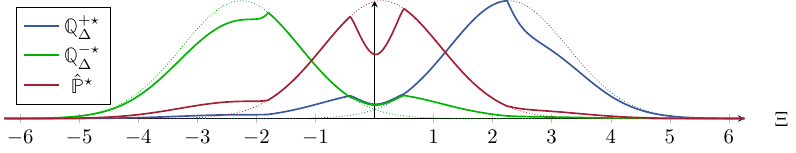}
    \caption{Worst-case distributions $\KLradp{r}{\alpha}<\KLrad{r}{\alpha}=\infty$}
    \label{fig:ex:eventual-distributions-iii}
  \end{subfigure}%
  \hfill
  \begin{subfigure}[b]{0.35\textwidth}
    \centering
      \includegraphics[height=5cm]{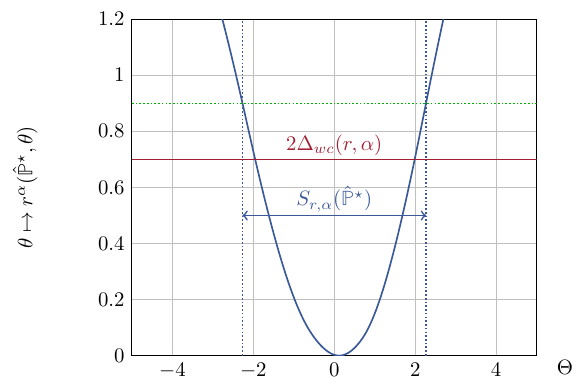}
    \caption{Resolution function $\KLradp{r}{\alpha}<\KLrad{r}{\alpha}=\infty$}
    \label{fig:ex:eventual-residual-iii}
  \end{subfigure}%
  \caption{Illustration of the worst-case distributions in the three different almost sure efficiency regimes.}
\end{sidewaysfigure}

\subsection{Regime $\KLradp{r}{\alpha}=\KLrad{r}{\alpha}<\infty$}

Perhaps the simplest situation is the regime in which and both radii are equal and take on a finite value. As before the Huber confidence set estimator of radius $\Delta>\KLrad{r}{\alpha}$ around the Huber estimator enjoys the statistical guarantee (\ref{eq:feasibility:parametric}).
However, we can consider for any $\Delta'<\KLradp{r}{\alpha}$ the event $\mb P_n\to \hat{\mb{P}}^\star$ for some $\hat{\mb{P}}^\star\in \mc P^\alpha_\Theta$ with $\Delta'\leq \radius(S_{r, \alpha}(\hat{\mb P}^\star))$ for which we get
\begin{align*}
  \radius(S^{\huber}_\Delta({\mb P}_n )) = \Delta >  \KLrad{r}{\alpha} & =\\
                                                    \KLradp{r}{\alpha}&\geq \radius(S_{r, \alpha}(\hat{\mb P}^\star))= \radius(\textstyle\lim_{\delta\downarrow 0} \lim_{n\to\infty}S^\delta_{r, \alpha}(\mb P^n))\geq \Delta'.
\end{align*}
Hence, as we may let $\Delta\downarrow \KLrad{r}{\alpha}$ and $\Delta'\uparrow\KLradp{r}{\alpha}=\KLrad{r}{\alpha}$, the Huber estimator is not only worst-case minimal but also almost sure minimal as an adversary may sample the data from the corrupted distribution $\mb P^c=\hat{\mb{P}}^\star\in \mc P^\alpha_\Theta$ in which case the event $\mb P_n\to \mb P^c=\hat{\mb{P}}^\star$ occurs almost surely \citep{varadarajan1958convergence}.

\begin{example}[Normal Family Continued]
  \label{ex:normal-family-3}
  Consider the standard normal location family $\mc P_\Theta = \set{N(\theta,1)}{\theta\in \Re}$ parametrized in the unknown mean $\theta\in\Re$ from Example \ref{ex:normal-family-2}. Let us consider the parameters $\alpha=0.1$ and $r=0.05$ indicated in Figure \ref{fig:diagram}. We numerically verify here that indeed $\KLradp{r}{\alpha}=\KLrad{r}{\alpha} = 0.66$ and depict the minimizers of  (\ref{eq:statistical-resolution-function-optimization-3}) in Figure \ref{fig:ex:eventual-distributions-i}. In Figure \ref{fig:ex:eventual-residual-i} we depict the confident set estimate $\SpKL{r}{\alpha}(\hat {\mb P}^\star)$ through its residual function as well and observe that in fact we have here the even stronger identity
  \[
    \SpKL{r}{\alpha}(\hat {\mb P}^\star) 
    =
    \set{\theta\in \Theta}{\abs{\theta-E^\star_\Delta(\hat {\mb P}^\star)}\leq \Delta}
  \]
  for $\Delta = \KLradp{r}{\alpha}=\KLrad{r}{\alpha}$.
\end{example}

\subsection{Regime $\KLrad{r}{\alpha}>\KLradp{r}{\alpha}$}

In general, however, we may have that $\KLrad{r}{\alpha}>\KLrad{r}{\alpha}'$ and we may write instead
\begin{align*}
  \radius(S^{\huber}_\Delta({\mb P}_n )) = \Delta >  \KLrad{r}{\alpha} & >\\
                                                    \KLradp{r}{\alpha}&\geq \radius(S_{r, \alpha}(\hat{\mb P}^\star))= \radius(\textstyle\lim_{\delta\downarrow 0} \lim_{n\to\infty}S^\delta_{r, \alpha}(\mb P^n))  > \Delta'.
\end{align*}
and hence in particular the Huber estimator is not almost sure minimal in this regime.
This means furthermore that although we may observe confidence set estimates up to radius $\KLrad{r}{\alpha}$, eventually in the large data limit, we fence in the unknown parameter within a confidence set of strictly smaller radius $\KLradp{r}{\alpha}$ using the minimal DRO confidence set estimator.

From \cref{thm:efficiency:location} we know that any fixed width confidence set based on a regular estimator must have a radius at least $\KLrad{r}{\alpha}$ to be feasible in (\ref{eq:feasibility:parametric}). As here $\KLradp{r}{\alpha}<\KLrad{r}{\alpha}$ this implies that no confidence set based on a regular estimator  can be minimal as pointed out in \cref{eq: KL radius prime} in terms of almost sure radius. In particular, while the confidence set estimator based on the Huber estimator is worst-case minimal it is not almost sure minimal in this regime.
This phenomena becomes particularly pronounced when we have $\KLradp{r}{\alpha}<\KLrad{r}{\alpha}=\infty$; see also Section \ref{sec:extreme-regime-case}. Here, the Huber confidence sets no matter their radius can not satisfy the imposed statistical guarantee (\ref{eq:feasibility:parametric}), whereas the minimal DRO confidence set estimator will almost surely fence in the unknown parameter eventually within a bounded set of radius at most $\KLradp{r}{\alpha}<\infty$.

{\color{black}
We have indicated that the radius worst-case $\KLrad{r}{\alpha}$ can be determined essentially analytically via \cref{eq:statistical-resolution-function-optimization} with the help of the least favorable distributions stated in \cref{lemma:huber-optimal-solution}. The situation is more complex for the almost sure radius $\KLradp{r}{\alpha}$ as we are not able to determine an analytical solution to problem (\ref{eq:statistical-resolution-function-optimization-3}). In the following examples in the context of a normal location family we determine the almost sure radius numerically instead by discretizing $\Xi$.
}

\begin{example}[Normal Family Continued]
  \label{ex:normal-family-4}
  Consider once more the standard normal location family $\mc P_\Theta = \set{N(\theta,1)}{\theta\in \Re}$ parametrized in the unknown mean $\theta\in\Re$ from Example \ref{ex:normal-family-2}.

  Let us now consider the parameters $\alpha=0.1$ and $r=0.4$ indicated in Figure \ref{fig:diagram}. We numerically verify here that indeed $1.48=\KLradp{r}{\alpha}<\KLrad{r}{\alpha} = 1.8$ and we depict the minimizers of (\ref{eq:statistical-resolution-function-optimization-3}) at $\Delta = \KLradp{r}{\alpha}$ in Figure \label{sec:regime-klradr}~\ref{fig:ex:eventual-distributions-ii}. Despite the fact that in this context while the Huber confidence estimator is minimal in terms of worst-case radius $\KLrad{r}{\alpha} = 1.8$ among any regular confidence set estimator, the DRO confidence estimator eventually fences in the unknown parameter within a set of strictly smaller radius $\KLradp{r}{\alpha}=1.48$ and should hence eventually be preferred; see also Figure \ref{fig:ex:eventual-residual-ii}.

  Let us finally consider $\alpha=0.1$ and $r=0.9$ indicated in Figure \ref{fig:diagram}. As this point is above the dotted feasibility line discussed in Section \ref{sec:extreme-regime-case} it follows $\KLrad{r}{\alpha} = \infty$. That is, no regular confidence set predictor can fence in the unknown location parameter within bounded set universally over all possible empirical distributions. However, we determine numerically that that $2.27=\KLradp{r}{\alpha}<\KLrad{r}{\alpha} = \infty$ and depict the minimizers of  (\ref{eq:statistical-resolution-function-optimization-3}) in Figure \ref{fig:ex:eventual-distributions-iii}. That is, even though not universally bounded, the DRO confidence set estimator will eventually bound the unknown location within a confidence set of radius $2.27=\KLradp{r}{\alpha}$; see also Figure \ref{fig:ex:eventual-residual-iii}.
\end{example}

\subsection{Regime $\KLradp{r}{\alpha}=\infty$}

We now point out that we may also have $\KLradp{r}{\alpha}=\infty$ in which case we can not even eventually fence in the unknown parameter within a bounded set.
In this regime the adversary the adversary can prevent any learning method from achieving the guarantee (\ref{eq:feasibility:parametric}).

To do so, for any $\Delta$, we will construct a feasible solution in Equation~(\ref{eq:statistical-resolution-function-optimization-3}) with a finite objective less than $\bar{r}'(\alpha)\defn \log\left(\tfrac{(1-\alpha)}{\alpha}\right) \left(1-\alpha\right) + \log\left(\tfrac{\alpha}{(1-\alpha)}\right)\alpha$. This implies $\rp^\alpha(\Delta) < \bar{r}(\alpha)$ for any $\Delta$, and therefore $\KLradp{r}{\alpha} = \infty$ for any $r \geq \bar{r}'(\alpha)$ using \cref{theorem:statistical-resolution-function-3}.
Let us introduce indeed the distributions
\begin{align*}
  \mb Q^-_{\Delta} = & (1-\alpha) \mb P_{-\Delta} + \alpha \mb P_{\Delta},\\
  \mb Q^+_{\Delta} = \hat{\mb P} = & (1-\alpha) \mb P_{\Delta} + \alpha \mb P_{-\Delta}.
\end{align*}
It can easily be verified that indeed $\TV(\mb P_{-\Delta}, \mb Q^-_{\Delta} ) = \TV(\mb P_{-\Delta}, (1-\alpha)\mb P_{-\Delta}+\alpha \mb P_{\Delta}) \leq (1-\alpha)\TV(\mb P_{-\Delta}, \mb P_{-\Delta})  +\alpha \TV(\mb P_{-\Delta}, \mb P_{\Delta}) \leq \alpha$ and symmetrically $\TV(\mb P_{\Delta}, \mb Q^+_{\Delta} )\leq \alpha$. Observe that
\begin{align*}
  \lim_{\Delta\to\infty} \KL((1-\alpha) \mb P_{\Delta} + \alpha \mb P_{-\Delta}, \mb Q^-_{\Delta}) = & \lim_{\Delta\to\infty} \KL((1-\alpha) \mb P_{\Delta} + \alpha \mb P_{-\Delta}, (1-\alpha) \mb P_{-\Delta} + \alpha \mb P_{\Delta})\\
  = & \log\left(\frac{1-\alpha}{\alpha}\right) \left(1-\alpha\right) + \log\left(\frac{\alpha}{1-\alpha}\right)\alpha = \bar r'(\alpha).
\end{align*}

Hence, \cref{theorem:statistical-resolution-function-3} implies the following result.

\begin{proposition}
    We have $\KLradp{r}{\alpha} = \infty$ for any $r\geq \bar r'(\alpha) = \log\left(\frac{1-\alpha}{\alpha}\right) \left(1-\alpha\right) + \log\left(\frac{\alpha}{1-\alpha}\right)\alpha$.
\end{proposition}

\bibliography{references}

\begin{thebibliography}{34}
\providecommand{\natexlab}[1]{#1}
\providecommand{\url}[1]{\texttt{#1}}
\expandafter\ifx\csname urlstyle\endcsname\relax
  \providecommand{\doi}[1]{doi: #1}\else
  \providecommand{\doi}{doi: \begingroup \urlstyle{rm}\Url}\fi

\bibitem[Huber(1964)]{huber1964robust}
Peter~J Huber.
\newblock Robust estimation of a location parameter.
\newblock \emph{The Annals of Mathematical Statistics}, pages 73--101, 1964.

\bibitem[Kleywegt and Shapiro(2001)]{kleywegt2001stochastic}
Anton~J Kleywegt and Alexander Shapiro.
\newblock Stochastic optimization.
\newblock \emph{Handbook of industrial engineering}, pages 2625--2649, 2001.

\bibitem[Kleywegt et~al.(2002)Kleywegt, Shapiro, and
  {Homem-de-Mello}]{kleywegtSampleAverageApproximation2002}
Anton~J. Kleywegt, Alexander Shapiro, and Tito {Homem-de-Mello}.
\newblock The {{Sample Average Approximation Method}} for {{Stochastic Discrete
  Optimization}}.
\newblock \emph{SIAM Journal on Optimization}, 12\penalty0 (2):\penalty0
  479--502, January 2002.
\newblock ISSN 1052-6234, 1095-7189.
\newblock \doi{10.1137/S1052623499363220}.

\bibitem[Smith and Winkler(2006)]{smith2006optimizer}
James~E Smith and Robert~L Winkler.
\newblock The optimizer’s curse: Skepticism and postdecision surprise in
  decision analysis.
\newblock \emph{Management Science}, 52\penalty0 (3):\penalty0 311--322, 2006.

\bibitem[Blanchet and Shapiro(2023)]{blanchetStatisticalLimitTheorems2023a}
Jose Blanchet and Alexander Shapiro.
\newblock Statistical {{Limit Theorems}} in {{Distributionally Robust
  Optimization}}, March 2023.

\bibitem[Bertsimas et~al.(2018)Bertsimas, Gupta, and Kallus]{bertsimas2018data}
Dimitris Bertsimas, Vishal Gupta, and Nathan Kallus.
\newblock Data-driven robust optimization.
\newblock \emph{Mathematical Programming}, 167:\penalty0 235--292, 2018.

\bibitem[Delage and Ye(2010)]{delage2010distributionally}
Erick Delage and Yinyu Ye.
\newblock Distributionally robust optimization under moment uncertainty with
  application to data-driven problems.
\newblock \emph{Operations research}, 58\penalty0 (3):\penalty0 595--612, 2010.

\bibitem[Lam(2019)]{lamRecoveringBestStatistical2019}
Henry Lam.
\newblock Recovering {{Best Statistical Guarantees}} via the {{Empirical
  Divergence-Based Distributionally Robust Optimization}}.
\newblock \emph{Operations Research}, 67\penalty0 (4):\penalty0 1090--1105,
  July 2019.
\newblock ISSN 0030-364X.
\newblock \doi{10.1287/opre.2018.1786}.

\bibitem[Duchi et~al.(2021)Duchi, Glynn, and Namkoong]{duchi2021statistics}
John~C Duchi, Peter~W Glynn, and Hongseok Namkoong.
\newblock Statistics of robust optimization: A generalized empirical likelihood
  approach.
\newblock \emph{Mathematics of Operations Research}, 46\penalty0 (3):\penalty0
  946--969, 2021.

\bibitem[Pardo(2018)]{pardo2018statistical}
Leandro Pardo.
\newblock \emph{Statistical inference based on divergence measures}.
\newblock Chapman and Hall/CRC, 2018.

\bibitem[Shafieezadeh-Abadeh et~al.(2019)Shafieezadeh-Abadeh, Kuhn, and
  Esfahani]{shafieezadeh2019regularization}
Soroosh Shafieezadeh-Abadeh, Daniel Kuhn, and Peyman~Mohajerin Esfahani.
\newblock Regularization via mass transportation.
\newblock \emph{Journal of Machine Learning Research}, 20\penalty0
  (103):\penalty0 1--68, 2019.

\bibitem[Blanchet et~al.(2019)Blanchet, Kang, and Murthy]{blanchet2019robust}
Jose Blanchet, Yang Kang, and Karthyek Murthy.
\newblock Robust wasserstein profile inference and applications to machine
  learning.
\newblock \emph{Journal of Applied Probability}, 56\penalty0 (3):\penalty0
  830--857, 2019.

\bibitem[Gao et~al.(2024)Gao, Chen, and Kleywegt]{gao2024wasserstein}
Rui Gao, Xi~Chen, and Anton~J Kleywegt.
\newblock Wasserstein distributionally robust optimization and variation
  regularization.
\newblock \emph{Operations Research}, 72\penalty0 (3):\penalty0 1177--1191,
  2024.

\bibitem[Blanchet et~al.(2022)Blanchet, Murthy, and Si]{blanchet2022confidence}
Jose Blanchet, Karthyek Murthy, and Nian Si.
\newblock Confidence regions in wasserstein distributionally robust estimation.
\newblock \emph{Biometrika}, 109\penalty0 (2):\penalty0 295--315, 2022.

\bibitem[Liu et~al.(2023)Liu, Van~Parys, and Lam]{liu2023smoothed}
Zhenyuan Liu, Bart~PG Van~Parys, and Henry Lam.
\newblock Smoothed $f$-divergence distributionally robust optimization:
  Exponential rate efficiency and complexity-free calibration.
\newblock \emph{arXiv preprint arXiv:2306.14041}, 2023.

\bibitem[Bertsimas and Popescu(2005)]{bertsimas2005optimal}
Dimitris Bertsimas and Ioana Popescu.
\newblock Optimal inequalities in probability theory: A convex optimization
  approach.
\newblock \emph{SIAM Journal on Optimization}, 15\penalty0 (3):\penalty0
  780--804, 2005.

\bibitem[Wiesemann et~al.(2014)Wiesemann, Kuhn, and
  Sim]{wiesemann2014distributionally}
Wolfram Wiesemann, Daniel Kuhn, and Melvyn Sim.
\newblock Distributionally robust convex optimization.
\newblock \emph{Operations research}, 62\penalty0 (6):\penalty0 1358--1376,
  2014.

\bibitem[Hu and Hong(2013)]{hu2013kullback}
Zhaolin Hu and L~Jeff Hong.
\newblock Kullback-leibler divergence constrained distributionally robust
  optimization.
\newblock \emph{Available at Optimization Online}, 1\penalty0 (2):\penalty0 9,
  2013.

\bibitem[Bayraksan and Love(2015)]{bayraksan2015data}
G{\"u}zin Bayraksan and David~K Love.
\newblock Data-driven stochastic programming using phi-divergences.
\newblock In \emph{The operations research revolution}, pages 1--19. INFORMS,
  2015.

\bibitem[Kuhn et~al.(2019)Kuhn, Esfahani, Nguyen, and
  Shafieezadeh-Abadeh]{kuhn2019wasserstein}
Daniel Kuhn, Peyman~Mohajerin Esfahani, Viet~Anh Nguyen, and Soroosh
  Shafieezadeh-Abadeh.
\newblock Wasserstein distributionally robust optimization: Theory and
  applications in machine learning.
\newblock In \emph{Operations research \& management science in the age of
  analytics}, pages 130--166. Informs, 2019.

\bibitem[Gao and Kleywegt(2023)]{gao2023distributionally}
Rui Gao and Anton Kleywegt.
\newblock Distributionally robust stochastic optimization with wasserstein
  distance.
\newblock \emph{Mathematics of Operations Research}, 48\penalty0 (2):\penalty0
  603--655, 2023.

\bibitem[Van~Parys et~al.(2021)Van~Parys, Esfahani, and Kuhn]{vanparys2021data}
Bart~PG Van~Parys, Peyman~Mohajerin Esfahani, and Daniel Kuhn.
\newblock From data to decisions: Distributionally robust optimization is
  optimal.
\newblock \emph{Management Science}, 67\penalty0 (6):\penalty0 3387--3402,
  2021.

\bibitem[Van~Parys(2024)]{vanparys2024efficient}
Bart~PG Van~Parys.
\newblock Efficient data-driven optimization with noisy data.
\newblock \emph{Operations Research Letters}, 54:\penalty0 107089, 2024.

\bibitem[Bennouna and Van~Parys(2022)]{bennouna2022holistic}
Amine Bennouna and Bart Van~Parys.
\newblock Holistic robust data-driven decisions.
\newblock \emph{arXiv preprint arXiv:2207.09560}, 2022.

\bibitem[Strassen(1965)]{strassen1965existence}
Volker Strassen.
\newblock The existence of probability measures with given marginals.
\newblock \emph{The Annals of Mathematical Statistics}, 36\penalty0
  (2):\penalty0 423--439, 1965.

\bibitem[Huber(1981)]{huber1981robust}
Peter~J. Huber.
\newblock \emph{Robust Statistics}.
\newblock Wiley series in probability and mathematical statistics. John Wiley
  \& Sons, 1981.

\bibitem[Diakonikolas and Kane(2023)]{diakonikolas2023algorithmic}
Ilias Diakonikolas and Daniel~M Kane.
\newblock \emph{Algorithmic high-dimensional robust statistics}.
\newblock Cambridge University Press, 2023.

\bibitem[Jiang and Xie(2021)]{jiang2021dfo}
Nan Jiang and Weijun Xie.
\newblock Dfo: A framework for data-driven decision-making with endogenous
  outliers.
\newblock \emph{Preprint optimization-online. org}, 2021.

\bibitem[Mosteller and Tukey(1977)]{mosteller1977data}
Frederick Mosteller and John~W Tukey.
\newblock Data analysis and regression. a second course in statistics.
\newblock \emph{Addison-Wesley series in behavioral science: quantitative
  methods}, 1977.

\bibitem[Dembo and Zeitouni(2009)]{dembo2009large}
Amir Dembo and Ofer Zeitouni.
\newblock \emph{Large deviations techniques and applications}.
\newblock Springer, 2009.

\bibitem[Yang and Chen(2018)]{yang2018robust}
Pengfei Yang and Biao Chen.
\newblock Robust kullback-leibler divergence and universal hypothesis testing
  for continuous distributions.
\newblock \emph{IEEE Transactions on Information Theory}, 65\penalty0
  (4):\penalty0 2360--2373, 2018.

\bibitem[Prokhorov(1956)]{prokhorov1956convergence}
Yu~V Prokhorov.
\newblock Convergence of random processes and limit theorems in probability
  theory.
\newblock \emph{Theory of Probability \& Its Applications}, 1\penalty0
  (2):\penalty0 157--214, 1956.

\bibitem[Bhattacharyya(1943)]{bhattacharyya1943measure}
Anil Bhattacharyya.
\newblock On a measure of divergence between two statistical populations
  defined by their probability distribution.
\newblock \emph{Bulletin of the Calcutta Mathematical Society}, 35:\penalty0
  99--110, 1943.

\bibitem[Varadarajan(1958)]{varadarajan1958convergence}
Veeravalli~S Varadarajan.
\newblock On the convergence of sample probability distributions.
\newblock \emph{Sankhy{\=a}: The Indian Journal of Statistics (1933-1960)},
  19\penalty0 (1/2):\penalty0 23--26, 1958.

\end{thebibliography}

\end{document}